\documentclass[12pt,a4paper]{article}
\usepackage{amsmath,amsthm,amsfonts,amssymb,color,soul}
  \definecolor{darkgreen}{rgb}{0,0.4,0}
  \definecolor{darkred}{rgb}{0.7,0,0}

\usepackage{tikz-cd}

\usepackage{url}

\parskip=\medskipamount
\def\onehalf{{\textstyle\frac12}}
\def\oneqtr{{\textstyle\frac14}}

\def\ov#1{\overline{#1}}

\def\lie#1{{\mathcal L}_{#1}}

\def\R{{\mathbb R}}

\def\SODE{{\textsc{sode}}}
\def\SODEs{{\textsc{sode}s}}
\def\hook{{\mathchoice{\vrule height 0pt depth 0.4pt width 3pt
\vrule height 5pt depth 0.4pt \kern 3pt} {\vrule height 0pt depth
0.4pt width 3pt \vrule height 5pt depth 0.4pt \kern 3pt} {\vrule
height 0pt depth 0.2pt width 1.5pt \vrule height 3pt depth 0.2pt
width 0.2pt \kern 1pt} {\vrule height 0pt depth 0.2pt width 1.5pt
\vrule height 3pt depth 0.2pt width 0.2pt \kern 1pt} }}

\def\d{\mbox{d}}
\def\tGamma{{\tilde\Gamma}}
\def\bGamma{{\overline\Gamma}}
\def\hGamma{{\hat\Gamma}}

\theoremstyle{plain}
\newtheorem{thm}{Theorem}[section]
\newtheorem*{thm*}{Theorem}
\newtheorem{lem}[thm]{Lemma}
\newtheorem{lem*}[thm]{Lemma}
\newtheorem{propn}[thm]{Proposition}
\newtheorem*{propn*}{Proposition}
\newtheorem{cor}[thm]{Corollary}

\theoremstyle{definition}
\newtheorem{defn}[thm]{Definition}
\newtheorem{xmpl}{Example}

\newtheorem*{xmpl*}{Example}
\newtheorem*{appln*}{Application}
\newtheorem*{comment*}{Comment}
\newtheorem*{comments*}{Comments}

\newcommand{\pd}[2]{\frac{\partial#1}{\partial#2}}
\newcommand{\vf}[1]{\frac{\partial}{\partial #1}}

\newcommand{\tH}{\tilde{H}}
\newcommand{\tV}{\tilde{V}}
\newcommand{\tE}{\tilde{E}}
\newcommand{\tS}{\tilde{S}}
\newcommand{\tF}{\tilde{\mathcal F}}
\newcommand{\tlF}{\tilde{F}}
\newcommand{\tOmega}{\tilde{\Omega}}

\newcommand{\iE}{i_{\tE}}
\newcommand{\PF}{{\iE}_*}
\newcommand{\PB}{{\iE}^*}
\newcommand{\tpsi}{\tilde{\psi}}
\newcommand{\tth}{\tilde{\theta}}
\newcommand{\teta}{\tilde{\eta}}

\newcommand{\Sp}{\text{Sp}} 
\newcommand{\Img}{\text{im}}

\begin{document}

\title{Variational Inverse Problems for Second Order ODEs with and without Constraints}

\author{G.E.\,Prince, T.\,Mestdag and D.\,Mart\'{i}n de Diego} 
\date{}
\maketitle

\begin{quote}

{\bf Abstract.} 
{\small 
Many physical systems with or without nonholonomic constraints have a Lagrangian description. In the first case, the Lagrangian model can be represented by second-order ODEs that are constrained to a submanifold of velocities; in the latter case the ODEs are unconstrained. In this paper,  using geometric techniques, we  address the more  general inverse problem: ``When can a given constrained or unconstrained system of second order ODEs on a manifold be the representation of  a Lagrangian model?''. We show that the constrained case has many more ambiguities and complexities than its well-understood, unconstrained counterpart.\\


}

\end{quote}


\section{Introduction and Summary of Results}

\subsection{The Inverse Problem in the Calculus of Variations}

 We began this work intending to apply the recent results on constrained systems by Prince, Farr\'{e} Puiggal\'{i}, Saunders and Mart\'{i}n de Diego 
 \cite{PPSD21} to the inverse problem 
 of finding Lagrangians for such systems. Along the way we discovered fresh features and utility in the geometric tools used in the standard, unconstrained 
 inverse problem in the calculus of 
 variations, see for example \cite{KP08,CPT84,MFLMR90}. As a result, this paper is concerned with the structure of the Cartan two-form of the calculus of variations 
 and its application to the variational inverse problems of both unconstrained and constrained systems of second order ordinary differential 
 equations, \SODEs\ for short. One of our goals is a foundation for the study of the inverse problem for nonholonomic systems akin to that of Crampin \cite{C81,C83,CPT84} for the unconstrained case.
 
In its twentieth century form the inverse problem in the calculus of variations is due to Jesse Douglas \cite{D41} who, in 1941, asked whether the 
solutions of a system of second order ODEs in normal form are exactly those of a system of Euler-Lagrange equations. His ambition was to classify all such 
\SODEs\ according to the uniqueness and existence of such regular Lagrangians. He gave a comprehensive  classification for a pair of equations. In her 2016 
PhD thesis \cite{Do16}, Thoan Do, gave us the most comprehensive outline of the solution for systems of $n$ equations, founded in large part on the results 
of numerous  workers over many decades (see \cite{KP08} for references and \cite{DP21} for a timeline). 

This inverse problem is formulated locally following Douglas and Sarlet \cite{S82}: 

  ``When are the solutions of
\[
\ddot{x}^a=F^a(t,x^b,\dot{x}^b)\ a,b=1,\dots,m
\] 
the solutions of 
\[
\frac{\partial^2L}{\partial \dot{x}^a \partial \dot{x}^b} \ddot{x}^b +\frac{\partial^2L}{\partial x^b 
\partial \dot{x}^a} \dot{x}^b + \frac{\partial^2L}{\partial t \partial \dot{x}^a } =\frac{\partial L}{\partial x^a}
\]
for some regular $L(t,x^a,\dot{x}^a)$?''

That is, do there exist non-singular  $g_{ab}(t,x^c,{\dot x}^c)$  (and $L$) so that 
\[
g_{ab} (\ddot{x}^b-F^b) \equiv \frac{d}{dt} \left(\frac{\partial L}{\partial \dot{x}^a}\right) 
-\frac{\partial L}{\partial x^a}?
\] 
Sarlet gives a linear differential-algebraic system of equations for the {\em multiplier}  $g_{ab},$ called the {\em Helmholtz conditions}.
 
 Geometric techniques have long been effective in the broad area of arbitrary but finite dimensional systems of second order equations. In these terms the question above has a positive answer if and only if there exists a certain maximal rank two-form on an odd-dimensional manifold 
 whose annihilator is a vector field representing the \SODEs\ \cite{CPT84}. We give this result with a new, self-contained proof in theorem \ref{CPT Thm(1)}. 
 To date the solution of Douglas's classification 
 problem in arbitrary dimension is most effectively obtained using Exterior Differential Systems theory (EDS), see \cite{AnTh92,A03,APST06,Do16,DP16,DP21}.
 
However, the calculus of variations is not appropriate to derive the equations of motion for systems with nonholonomic 
constraints because the constraints produce reaction forces \cite{Bloch03,Cortes2002,NeiFuf72}. That is, this motion is not described by classical variational principles. In these cases the corresponding equations of motion are derived by Lagrange-d’Alembert’s or Chetaev Principles. And so we enter new territory when we extend the inverse problem to \SODEs\ with first order constraints:
 
``When are the solutions of
\begin{equation*}
\ddot x^a = \tilde F^a (t, x^b,x^\beta, \dot x^b), \quad
\dot x^\alpha =\Psi^\alpha(t,x^b,x^\beta,\dot x^b)\quad a=1,\dots,m;\ \alpha=m+1,\dots,n
\end{equation*}
the solutions of the Lagrange-d'Alembert equations or Chetaev equations \cite{Bloch03}:
\begin{equation*}
\frac{d}{dt}\left(\pd{L}{\dot x^a}\right)-\pd{L}{x^a}=-\lambda_\alpha \pd{\Psi^\alpha}{\dot x^a},\qquad
\frac{d}{dt}\left(\pd{L}{\dot x^\alpha}\right)-\pd{L}{x^\alpha}=\lambda_\alpha,\qquad
\dot x^\alpha=\Psi^\alpha
\end{equation*}
for some Lagrangian $L(t,x^A,\dot{x}^B)$ and Lagrange multipliers $\lambda_\alpha?$'' (Here $A,B=1,\dots,n.$) 

What might appear to be a heavily restricted inverse problem actually admits greater multiplicity of solutions than the unconstrained one for reasons we 
indicate in the next subsection.
Moreover, the problem is of particular interest as it encompasses a framework more general than that of the inverse problem in the calculus of variations. In this context, given a \SODE\ on a submanifold, we examine whether it can be interpreted as the dynamics of a physical system, understood as one whose trajectories are governed by the Lagrange–d’Alembert principle.

Here is an outline of the paper. The main results include a new elaboration of the Cartan two-form in terms of the geometric structures 
(horizontal distributions, etcetera) associated with an arbitrary second order system; its use to give fresh insights and a new proof of the Crampin, 
Prince and Thompson theorem \cite{CPT84}; a geometric formulation of nonholonomic constraints, the Lagrange-d'Alembert equations and the corresponding inverse problem along with 
its regularity, Helmholtz conditions and existence theorem. There follows an examination of the existence of classical mechanical Lagrangians for constrained, unconstrained, 
and dissipative systems along with some examples from physics. Finally, we give a brief idea of future directions. \\ 

\subsection{Motivation}\label{motivation}
In the unconstrained inverse problem for a given \SODE\  we find {\em equivalent} Lagrangians, meaning that they all have the same Euler-Lagrange 
equations, equivalent to the \SODE. In the constrained inverse problem, for given constraints and a constrained \SODE, all the Lagrangians we find will 
have the same Lagrange-d'Alembert equations, however they may not have the same Euler-Lagrange field. That is, distinct unconstrained \SODEs\ may produce 
the same constrained equations. In physical terms, for fixed constraints different external forces may produce the same constrained dynamics, meaning that 
an unknown external force may not be able to be identified from observed constrained behaviour. From a mathematical perspective this non-uniqueness is to be expected and it creates an additional classification opportunity. 

Here is a physical illustration. Consider the free particle on $\R^3$ with the standard inner product:
\[
\ddot x=0,\ \ddot y=0,\ \ddot z=0.
\]
For the purpose of this introduction, we will restrict ourselves to finding physical or `magnetic' Lagrangians of the type
\begin{equation}\label{xmpl 1.1}
L=\onehalf({\dot x}^2+{\dot y}^2+{\dot z}^2)-V(x^B) + M_A(x^B){\dot x}^A.
\end{equation}
It is straightforward to show that the only such Lagrangians for this  free particle are $L=\onehalf({\dot x}^2+{\dot y}^2+{\dot z}^2)+\pd{M}{x^A}{\dot x^A},$ 
that 
is, the kinetic energy Lagrangian plus a total time derivative.\\
Now consider the same system, with Lagrangian $L=\onehalf({\dot x}^2+{\dot y}^2+{\dot z}^2),$ in the presence of the constraint $\dot z=x\dot y.$ The 
resulting Lagrange-d'Alembert equations for this so-called {\em nonholonomic free particle} are
\[
\ddot x=0,\ \ddot y=-\frac{x\dot x\dot y}{1+x^2},\ \dot z=-x\dot y,
\]
with initial conditions satisfying the contraint
(see e.g.\ \cite{Cortes2002}) and if we ask for other Lagrangians of the physical type above which produce these constrained equations from the same constraint, this time we find $V=0$ as 
before but   the vector potential term only has to satisfy 
\[
\pd{M_y}{x}-\pd{M_x}{y}=x\left(\pd{M_z}{x}-\pd{M_x}{z}\right).
\] 
For example, with $M_x=xy+z$ and $M_y,\ M_z$ constants we have the Lagrangian
\begin{equation}\label{xmpl 1.2}
L=\onehalf({\dot x}^2+{\dot y}^2+{\dot z}^2) + (xy+z){\dot x} + c_1\dot y+c_2\dot z. 
\end{equation}
with Euler-Lagrange equations  
\[
\ddot x=-x\dot y-\dot z,\ \ddot y=x\dot x,\ \ddot z=\dot x,
\]
which indicate a non-trivial force. Nonetheless, for the constraint $\dot z=x\dot y,$ both the kinetic energy Lagrangian and \eqref{xmpl 1.2} produce the same constrained equations via the Lagrange-d'Alembert equations (but with different Lagrange multipliers). So this non-uniqueness in the constrained inverse problem is not just a non-physical mathematical artifact.

 \section{Constrained and unconstrained SODEs}
 
We follow the geometric description of \SODEs\ in \cite{S89,KP08,PPSD21} and references therein. \\
We deal with a manifold $M^n$ with generic local coordinates $(x^A),$ often called the {\em configuration space},  and with
associated bundles $\pi:\R \times M \rightarrow M,$ $t:\R \times M \rightarrow
\R$ and $\pi^0_1:E \rightarrow \R \times M$. Almost all results are local with no underlying metric or topological assumptions unless made explicit. The {\em evolution space} $E:=\R \times TM$ has adapted coordinates $(t,x^A,\dot x^A)$ or $(t,x^A,u^A)$ and is identified with $J^1(\R,M)$, the bundle of 1-jets of maps $\R\to M$. 

The {\em contact (co-)distribution} on $J^1(\R,M)$ has local description
\[
\Omega^1(\R,M)=\Sp\{\theta^A:=dx^A-u^Adt\}.
\]
If $J^2(\R,M)$ has adapted coordinates $(t,x^A,u^A,v^A)$ or $(t,x^A,\dot x^A, \ddot{x}^A)$ then  its contact distribution is 
\[
\Omega^2(\R,M)=\Sp\{\theta^A:=dx^A-u^Adt,\varphi^A:=du^A-v^Adt\}.
\]

The important geometric feature being that the lifts of graphs (1-graphs) of functions $f:\R\to M$ lie in the integral submanifolds of $\Omega^1(\R,M),$ and similarly for $\Omega^2(\R,M).$

The vertical sub-bundle $V(E)$ consists of the vertical subspaces of the tangent spaces of $E$, at each point $p\in E$, that is, it is the kernel
of ${(\pi^0_1)}_{\ast}:T_pE\to T_{\pi^0_1(p)}(\R\times M).$ The {\em vertical distribution} on $E$ is generated by $ \{V_A:=\frac{\partial}{\partial u^A}\}.$

The vertical and contact structures on $E$ are combined in the {\em vertical endomorphism}
$$ S:=\theta^A\otimes V_A.$$
An intrinsic definition of $S$ can be found in \cite{C83, CPT84,S89}.

\subsection{Unconstrained SODEs}\label{uncontr sodes}

In this section we are dealing with a system of smooth second-order ordinary differential equations on $M,$
\begin{align}\label{SODE1}
\ddot x^A = F^A (t, x^B, \dot x^B), \ \ A, B = 1, \dots, n.
\end{align}

These equations are identified with the codimension $n$  embedded submanifold of $J^2(\R,M):$
\[
\mathcal F:=\{q\in J^2(\R,M):v^A(q)-F^A(\pi^1_2(q))=0\},
\]
where $\pi^1_2: J^2\pi\rightarrow J^1\pi$.
If $i_{\mathcal F}$ is the inclusion map $\mathcal F \hookrightarrow J^2(\R,M)$, then the contact distribution restricted to $\mathcal F$ is
$$
i_{\mathcal F}^\ast \Omega^2(\R,M)=\Sp\{\theta^A:=dx^A-u^Adt, \phi^B:=du^B-F^Bdt\}
$$
with annihilator generated by the  {\em second-order differential equation field} (\SODE),
\begin{equation}\label{SODE2}
\Gamma := \frac{\partial}{\partial t} + u^A
\frac{\partial}{\partial x^A} + F^B \frac{\partial}{\partial u^B}
\end{equation}
whose integral curves are the lifted graphs of solution curves to \eqref{SODE1}.

 The first order deformation $\mathcal{ L}_\Gamma S$ is of natural interest (see \cite{C83,CPT84}), it has eigenvalues $0, 1$ and $-1$  with corresponding eigenspaces $\Sp\{\Gamma\}$, the {\em vertical distribution} $V(E)=\Sp\{V_A\}$ and
 the {\em horizontal distribution} $H(E)=\Sp\{H_A\}$ respectively, where
 $$
H_A := \frac{\partial}{\partial x^A} -\Gamma^B_A\frac{\partial}{\partial u^B},\ \Gamma_A^B := -{\frac{1}{2}}\frac{\partial F^B}{\partial u^A}.
$$
The resulting direct sum decomposition of the tangent space of $E$ gives an adapted local basis $\{\Gamma, V_A, H_A\}$ with
dual basis $\{dt,\psi^A,\theta^A\}$ where
 $$
\psi^A:=\phi^A+\Gamma^A_B\theta^B.
 $$
 The corresponding projectors  are denoted ${P_\Gamma, P_V} $ and $P_H$.

The bracket relations will be important:
\[
[\Gamma, H_A]=\Gamma^B_AH_B + \Phi^B_AV_B, \quad [\Gamma,
V_A]=-H_A + \Gamma^B_AV_B, \quad [V_A,V_B]=0,
\]
\[
[H_A, V_B]=-\frac{1}{2}\left(\frac{\partial^2 F^C}{\partial u^A\partial
u^B}\right)V_C=V_B(\Gamma^C_A)V_C=V_A(\Gamma^C_B)V_C=[H_B, V_A],
\]
and
\[
[H_A,H_B]=R^C_{AB} V_C,
\]
where
\[
\Phi_B^A := -\frac{\partial F^A}{\partial x^B} -
\Gamma_B^C\Gamma_C^A - \Gamma(\Gamma_B^A).
\]
The {\em Jacobi endomorphism} is  $\Phi=P_V\circ\lie{\Gamma}P_H =\Phi^A_B\theta^B\otimes V_A,$
and the curvature of the `non-linear connection' is 
$$
R^C_{AB}:=\frac{1}{2}\left (\frac{\partial^2 F^C}{\partial
x^A\partial u^B} - \frac{\partial^2 F^C}{\partial x^B\partial u^A}
+\frac{1}{2}\left (\frac{\partial F^D}{\partial
u^A}\frac{\partial^2 F^C}{\partial u^D\partial u^B} -
\frac{\partial F^D}{\partial u^B}\frac{\partial^2 F^C}{\partial
u^D\partial u^A} \right )\right ).
$$
Further details can be found in \cite{PPSD21} and its references.

\subsection{Constrained SODEs}

We start with the constrained system on $E:=\R\times TM$, 
\begin{equation}\label{basic system}
\ddot x^a = \tilde F^a (t, x^b,x^\beta, \dot x^b), \quad
\dot x^\alpha =\Psi^\alpha(t,x^b,x^\beta,\dot x^b)
\end{equation}
with $a,b=1,\dots,m$ and $\alpha,\beta=m+1,\dots,n.$  

Our extended configuration space is $\R\times M$ with coordinates $(t,x^a,x^\alpha)$ and for technical reasons we give $M$ a product structure: $M:=M^m\times M^{n-m}.$ The evolution space, $E$,  has adapted
coordinates $(t,x^a,x^\alpha,u^a,u^\alpha)$; we will use the combined index $A=(a,\alpha)$ for compatibility with the unconstrained case. 

Recall from \cite{PPSD21} that  $\tF$ is the embedded submanifold of $J^2(\R,M)$ of co-dimension~$n$
\[
\tF:=\{q\in J^2(\R,M):v^a(q)-\tilde F^a(q)=0, u^\alpha(q)-\Psi^\alpha(q)=0\}
\]
with inclusion map $i_{\tF}:\tF \hookrightarrow J^2(\R,M).$ $i_{\tF}$ is a section of $\tilde\pi^1_2:\tF\to \tilde E$ where $\tE$ is the embedded constraint
submanifold
\[
\tE:=\{p\in E: G^\alpha(p):=u^\alpha(p)-\Psi^\alpha(p)=0\}
\]
with inclusion map $\iE:\tE \hookrightarrow J^1(\R,M)$ (we use $G^\alpha:=u^\alpha-\Psi^\alpha$ throughout).  On $\tE$ the {\em constrained SODE}
\begin{equation}\label{tGamma}
\tGamma=\vf{t}+u^a\vf{x^a}+\Psi^\alpha\vf{x^\alpha}+\tilde F^a\vf{u^a}
\end{equation}
represents the constrained system \eqref{basic system}.

We locate both the restricted contact distribution, $\Sp\{\tth^a:=dx^a-u^adt,\tth^\alpha:=dx^\alpha-\Psi^\alpha dt\},$ and $\tGamma$ on $\tilde E$ and use local coordinates $(t,x^A,u^a)$ on $\tilde E.$
Because of the embedding of $\tilde E$ into $E$, at every point $p\in \tilde E$
we have  $\Sp\{\tilde V_a:=\pd{}{u^a}\}
\subset T_p\tilde E$, and we denote the corresponding `vertical' sub-bundle $\tilde V(\tilde E)$.

Since we can identify $\tilde{E}={\mathbb R}\times TM^m\times M^{n-m}$,
the adapted vertical endomorphism is $\tilde{S}=\tth^a\otimes \tV_a$; it is tensorial under coordinate transformations respecting the product structure of $M.$
$\lie{\tGamma}\tilde{S}$ has eigenvalues $0,1$ and $-1$ as before but now the eigendistributions are respectively,
$\Sp\{\tGamma,\pd{}{x^\alpha}\}$, $\Sp\{\tV_a\}$ and $\Sp\{\tilde H_b\}$, where
$$
\tilde H_a:=\frac{\partial}{\partial x^a} -\tGamma^b_a\frac{\partial}{\partial u^b}-\Psi^\beta_a\pd{}{x^\beta}
$$
with
\[
\tGamma_a^b:= -{\frac{1}{2}}\frac{\partial \tilde F^b}{\partial u^a}\ \text{and}\ \Psi^\beta_a:=-\pd{\Psi^\beta}{u^a}.
\]
These eigendistributions give a direct sum decomposition of $T\tilde E$ with basis $\{\tGamma,\tH_a,\vf{x^\alpha},\vf{u^a}\}$.
The dual basis is $\{dt, \teta^\alpha,\tpsi^a, \tth^b\}$ where 
$$
\teta^\alpha:=\tth^\alpha+\Psi^\alpha_b\tth^b,\ \ \tpsi^a:=(\d u^a-\tilde{F}^a\d t) + \tGamma^a_b\tth^b.
$$
Note: We make minor notational amendments to \cite{PPSD21} putting a tilde on the contact forms, $\theta^a,\theta^\alpha,$ on the embedded constraint
manifold, $\tE,$ to distinguish them from those on $E$; putting a tilde on $K^\alpha_a$ below, and putting a tilde on $\eta^\alpha$ and using $\eta^\alpha$ 
(without tilde) for the contact form on the image of $\iE$ which pulls back to $\teta^\alpha$.

The non-zero brackets of the basis fields are (compare with the unconstrained case): 
\begin{align}
\notag [\tGamma,\tilde H_a]&=\tilde \Phi^b_a\tV_b+\tGamma^b_a\tilde H_b+\tilde{K}^\alpha_a\pd{}{x^\alpha},
&&[\tGamma, \tV_a]=-\tilde H_a+\tGamma^b_a\tV_b,\\
[\tilde H_a,\tV_b]&=\pd{\tGamma^c_a}{u^b}\tV_c+\pd{\Psi^\alpha_a}{u^b}\pd{}{x^\alpha} = [\tilde H_b,\tV_a],
&&[\tilde H_a,\tilde H_b]=\hat R^c_{ab}\tV_c+\check R^\beta_{ab}\pd{}{x^\beta},\label{basis brackets}\\
\notag[\tGamma,\pd{}{x^\alpha}]&=-\pd{\Psi^\beta}{x^\alpha}\pd{}{x^\beta}-\pd{\tilde F^c}{x^\alpha}\tV_c,
&&[\tilde H_a,\pd{}{x^\alpha}]=\pd{\tGamma^b_a}{x^\alpha}\tV_b+\pd{\Psi^\beta_a}{x^\alpha}\pd{}{x^\beta}
\end{align}
where
\begin{align*}
\tilde\Phi^b_a&:=-\pd{\tlF^b}{x^a}-\tGamma(\tGamma^b_a)-\tGamma^c_a\tGamma^b_c+\Psi^\alpha_a\pd{\tlF^b}{x^\alpha},\\
\tilde{K}^\alpha_a&:=-\tGamma(\Psi^\alpha_a)-\tilde H_a(\Psi^\alpha)+\tGamma^b_a\Psi^\alpha_b,\\
R^d_{ab}&:=\frac{1}{2}\left (\frac{\partial^2 \tlF^d}{\partial x^a\partial u^b} -
\frac{\partial^2 \tlF^d}{\partial x^b\partial u^a}+\frac{1}{2}\left (\frac{\partial \tlF^c}{\partial u^a}\frac{\partial^2 \tlF^d}{\partial u^c\partial u^b}
-\frac{\partial \tlF^c}{\partial u^b}\frac{\partial^2 \tlF^d}{\partial u^c\partial u^a} \right )\right ),\\
\hat R^c_{ab}&:=R^c_{ab}- \Psi^\beta_b \pd{\tGamma^c_a}{x^\beta}+\Psi^\beta_a\pd{\tGamma^c_b}{x^\beta},\\
\check R^\beta_{ab}&:=\left (\pd{\Psi^\beta_a}{x^b}-\pd{\Psi^\beta_b}{x^a}\right )+\left( \tGamma^c_a\pd{\Psi^\beta_b}{u^c}
-\tGamma^c_b\pd{\Psi^\beta_a}{u^c} \right) +\left( \Psi^\alpha_a \pd{\Psi^\beta_b}{x^\alpha} - \Psi^\alpha_b \pd{\Psi^\beta_a}{x^\alpha} \right).
\end{align*}

For our later use:
\begin{equation*}
\PF\vf{t}=\pd{}{t}+\pd{\Psi^\alpha}{t}\vf{u^\alpha},\
\PF\vf{x^A}=\vf{x^A}+\pd{\Psi^\alpha}{x^A}\vf{u^\alpha},\
\PF\vf{u^a}=\vf{u^a}+\pd{\Psi^\alpha}{u^a}\vf{u^\alpha},
\end{equation*}

so that
\begin{align*}
&\PF\tGamma=\vf{t}+u^a\vf{x^a}+\Psi^\alpha\vf{x^\alpha}+\tilde F^a\vf{u^a}+\tGamma(\Psi^\alpha)\vf{u^\alpha},\\
&\PF\tH_a=\vf{x^a}-\tGamma^b_a\vf{u^b}+\pd{\Psi^\beta}{u^a}\vf{x^\beta}+\tH_a(\Psi^\alpha)\vf{u^\alpha},\\
&\PF{X}=X+X(\Psi^\alpha)\vf{u^\alpha};
\end{align*}
and
\begin{align*}
&\PB\d t=\d t,\quad\ \PB\theta^A=\tth^A,\quad\ \PB\eta^\alpha=\teta^\alpha,\quad \PB\d u^a=\d u^a,\\
&\PB\d u^\alpha=\d\Psi^\alpha =\pd{\Psi^\alpha}{t}\d t+\pd{\Psi^\alpha}{x^A}\d x^A+\pd{\Psi^\alpha}{u^a}\d u^a,\quad\ \PB\d G^\alpha=0.
\end{align*}
Hence, for any function $K$ on $E$
\[
\PB\d K=\overline{\d K}(\text{mod}\ \d u^\alpha)+\overline{\pd{K}{u^\alpha}}\d\Psi^\alpha,
\]
where, here and elsewhere, the overline indicates evaluation on the constraint submanifold, $\tE$, of $E$ by replacement of all occurrences of $u^\alpha$
by $\Psi^\alpha,$ that is, $\overline{K}=\PB K.$ Here this means that
\[
\ov{\d K}(\text{mod}\ \d u^\alpha)=\ov{\pd{K}{t}}\d t+\ov{\pd{K}{x^A}}\d x^A+\ov{\pd{K}{u^a}}\d u^a
\]
can be regarded as a form on $E$ or $\tE$.

Finally, for each \SODE\ $\Gamma$ and given constraint functions, $G^\alpha:=u^\alpha-\Psi^\alpha$, there exists a unique constrained \SODE\ $\hGamma$ 
with the property $\PF\hGamma=\Gamma-\overline{\Gamma(G^\alpha)}V_\alpha$ on the image of $\iE$. Explicitly, if
\[ 
\Gamma=\frac{\partial}{\partial t} + u^A\frac{\partial}{\partial x^A} + F^B \frac{\partial}{\partial u^B}
\]
then
\begin{equation}\label{hGamma}
\hGamma=\vf{t}+u^a\vf{x^a}+\Psi^\alpha\vf{x^\alpha}+\overline F^a\vf{u^a}
\end{equation}
with
\[
\PF\hGamma=\vf{t}+u^a\vf{x^a}+\Psi^\alpha\vf{x^\alpha}+\overline F^a\vf{u^a}+\hGamma(\Psi^\alpha)\vf{u^\alpha}.
\]
Notice that the condition $\overline{\Gamma(G^\alpha)}=0$ for all $\alpha,$ that is, $\Gamma(G^\alpha$)=0 when $G^\alpha=0$ for all $\alpha,$ means that $\Gamma$ is tangent to $\Img(\iE)$, so that $\hGamma=\Gamma$ on this image.

{\bf Constraint types and their integrability}

We have chosen to follow \cite{PPSD21} in treating the constraints together with the second order equations, rather than first treating  
the geometry of constraints, for example, as in \cite{SCS99}. 
However, we can now discuss the role of the constraints in the curvatures displayed in the bracket relations \eqref{basis brackets}. In \cite{SCS99}, the authors show that generic constraints manifest as the co-distribution $\Sp\{\teta^\alpha\}$ on $\tE.$ Now, in mechanics constraints are typically affine in the $u^a$ and either holonomic (`integrable') or nonholonomic; so we will deal first with the Frobenius integrability of the constraint co-distribution and then with its possibly affine character.

The constraint co-distribution annihilates $\Sp\{\tGamma,\tH_a,\tV_b\}$ and the presence of the vertical generators mean that this distribution is independent of the choice of $\tGamma.$ The bracket relations \eqref{basis brackets} then indicate that the constraints are Frobenius integrable only if $\tilde{K}^\alpha_a,\frac{\partial^2\Psi^\alpha}{\partial{u^a}\partial{u^b}},\check R^\beta_{ab}$ are all zero. We see from the expressions for $\check R$ and $\tilde{K}$ above that a necessary but not sufficient condition for Frobenius integrability is that the constraints be affine in the $u^a.$ As a result, non-affine constraints are never integrable whereas affine constraints can be either. The question of the independence of the three conditions can be examined using the Jacobi identities on $\{\tGamma,\tH_a,\tV_b\}$; it will be important later that the $\check R=0$ condition is implied by the other two, hence these two alone imply integrable constraints.  We will generally be concerned here with non-integrable constraints both affine and nonlinear. On the matter of deciding integrability in any particular case, the simplest approach is to test the condition $\d\teta^\alpha=\xi^\alpha_\beta\wedge\teta^\beta$.

Since 
\begin{align*}
\teta^{\alpha}&=\d x^{\alpha}-\Psi^{\alpha}\d t-\frac{\partial\Psi^\alpha}{\partial u^b}(\d x^b-u^b\d t)\\
&=\d x^{\alpha}-\frac{\partial \Psi^\alpha}{\partial u^b}\d x^b-(\Psi^{\alpha}-\frac{\partial\Psi^\alpha}{\partial u^c}u^c)\d t,
\end{align*}
the  integrability condition   $\d\teta^\alpha=\xi^\alpha_\beta\wedge\teta^\beta$
gives us the following condition: 
\begin{align*}
\frac{\partial^2\Psi^\alpha}{\partial u^a\partial u^b}&=0 
\end{align*}
which immediately implies that the constraints are affine on the velocities, that is: 
\[
\Psi^{\alpha}(t, x^a, x^{\beta}, u^a)=A^{\alpha}_a(t, x^a, x^{\beta})u^a+B^{\alpha}(t, x^a, x^{\beta})
\]
so that 
\[
\teta^\alpha=\d x^{\alpha}-A^{\alpha}_a \d x^a-B^{\alpha}\d t.
\]
We obtain the following additional conditions  for the integrability of the distribution: 
\begin{align*}
\frac{\partial A^{\alpha}_a}{\partial x^{\beta}} A^{\beta}_b-\frac{\partial A^{\alpha}_b}{\partial x^{\beta}} A^{\beta}_a+\frac{\partial A^{\alpha}_a}{\partial x^b}-\frac{\partial A^{\alpha}_b}{\partial x^a}&=0\\
\frac{\partial A^\alpha_a}{\partial t}+\frac{\partial A^{\alpha}_a}{\partial x^{\beta}}B^\beta-
\frac{\partial B^\alpha}{\partial x^a}-\frac{\partial B^\alpha}{\partial x^\beta}A^\beta_a&=0
\end{align*}
An example where these conditions are satisfied is when the constraints are of the form
\[
\Psi^{\alpha}=\frac{\partial f^{\alpha}}{\partial x^a}(t, x^b)u^a+\frac{\partial f^{\alpha}}{\partial t}(t, x^b)
\]
where $f^{\alpha}: {\mathbb R}\times M^m\rightarrow {\mathbb R}$. These constraints are the total derivative of the holonomic constraints
\[
x^\alpha=f^\alpha(t, x^b)\; .
\]
But other solutions are possible, as for instance:
\[
\Psi(t, x^a, x^{\alpha}, \dot{x}^a)=k_1\left[e^{k_2\sum_{b=1}^m x^b}+k_3\sum_{\beta=m+1}^{n} x^{\alpha}\right](\sum_{a=1}^m \dot{x}^a )+e^{k_1k_3\sum_{b=1}^m x^b}
\]
assuming a single constraint.

\section{Variationality} \label{sec3}

The Euler-Lagrange equations in unconstrained settings are traditionally derived from the variational principle as the system of second order differential
equations in covariant form
\begin{equation}\label{trad EL}
\frac{d}{dt}\left(\pd{L}{\dot x^A}\right)-\pd{L}{x^A}=0,
\end{equation}
where $L$ is the Lagrangian, assumed regular (that is, the {\em Hessian} matrix with components $\frac{\partial^2L}{\partial \dot x^A\partial \dot x^B}$
is locally
invertible), and $\frac{d}{dt} $ is the {\em total time derivative} \cite{S89}
\[
\frac{d}{dt}:=\pd{}{t}+\dot x^B\pd{}{x^B}+\ddot x^B\pd{}{\dot x^B}.
\]
If $\Gamma_L$ is the \SODE\  corresponding to \eqref{trad EL} cast in normal form, then we have the {\em identity}
\begin{equation}\label{EL id}
\Gamma_L\left(\pd{L}{u^A}\right)-\pd{L}{x^A}=0.
\end{equation}
$\Gamma_L$ is called the {\em Euler-Lagrange field} of $L$.

A simple statement of the inverse problem of the calculus of variations is then 
\begin{quote}
 {\em ``For a given \SODE\  $\Gamma$ find all regular functions $L$
satisfying the identity \eqref{EL id}"}.
\end{quote}
The related (Douglas) classification problem is  
\begin{quote}
{\em ``Classify \SODEs\ according to the existence and uniqueness of such functions." }
\end{quote}

Now we turn to constrained variational systems described by the Lagrange-d'Alembert principle. Suppose that $L$ is a regular Lagrangian on the full space
$E,$ and that the system has first order constraints described earlier. Then the constrained motion is described by the Lagrange-d'Alembert equations \cite{Bloch03, Cortes2002, LMM1997,NeiFuf72,SCS95}

\begin{equation}\label{LAeqn}
\frac{d}{dt}\left(\overline{\pd{L}{\dot x^a}}\right)-\overline{\pd{L}{x^a}}=-\lambda_\alpha \pd{\Psi^\alpha}{\dot x^a},\qquad
\frac{d}{dt}\left(\overline{\pd{L}{\dot x^\alpha}}\right)-\overline{\pd{L}{x^\alpha}}=\lambda_\alpha,\qquad
\dot x^\alpha=\Psi^\alpha.
\end{equation}
The $\lambda_\alpha$ are called {\em Lagrange multipliers} and $\frac{d}{dt}$ is the total time derivative as before. 

Under a further regularity condition on the Lagrangian detailed later in proposition \ref{Prop2.1}, these equations can be described by a single vector field which encodes them in normal form. This {\em
Lagrange-d'Alembert field}, $\tGamma_L,$ on the constraint submanifold, $\tE,$ is
\begin{equation}\label{tGamma LdA}
\tGamma_L=\vf{t}+u^a\vf{x^a}+\Psi^\alpha\vf{x^\alpha}+\tilde F^a\vf{u^a},
\end{equation}
where the $\tilde F^a$ are the values of the `accelerations' derived from \eqref{LAeqn}. $\tGamma_L$ satisfies the identity
\begin{equation}\label{LdA id}
\tGamma_L\left(\overline{\pd{L}{u^a}}\right)-\overline{\pd{L}{x^a}}=-\lambda_\alpha \pd{\Psi^\alpha}{u^a},
\end{equation}
where the overline indicates evaluation on $\tE$ by replacement of all occurrences of $u^\alpha$ by
$\Psi^\alpha.$

Note that $\tGamma_L$ {\em is not} the restriction to the constraint submanifold of the Euler-Lagrange field of $L$ on the full space. In other words, the
system is not the geometric superposition of the unconstrained system and the constraints, rather the Lagrange-d'Alembert principle encodes the interaction
of the external forces with the resulting reaction forces to the constraint submanifold.

A simple statement of the constrained inverse problem of the calculus of variations,   in which {\em regular} refers to both types discussed above,  is then \begin{quote} 
 {\em ``For a given constrained \SODE\  $\tGamma$ given by \eqref{tGamma} find all regular functions $L$
satisfying the identity \eqref{LdA id} with Lagrange multipliers given by the second of \eqref{LAeqn}"}
\end{quote}
And, in light of our comments in section 
\ref{motivation}, the related classification problem is 
\begin{quote}
{\em ``Classify constraints and constrained \SODEs\ according to the uniqueness and existence of such functions {\bf and} their Euler-Lagrange fields''}.
\end{quote} 
We will address this inverse problem in detail but we do not give a Douglas-type classification.

\subsection{SODEs and Cartan forms}
\subsubsection{The unconstrained case}

The key constructs used in the geometric treatment of these inverse problems are the Cartan one-form and its exterior derivative, the Cartan two-form (see e.g. \cite{KKS10} for some historical notes).

\begin{defn}\label{Cartan form}
Suppose that $K: E={\mathbb R}\times TM\rightarrow {\mathbb R}$ is a smooth function with regular Hessian on (some open subset of) $E$. The one-form
\[ \theta_K:=\d K\circ S+K\d t \]
is called the {\em Cartan one-form} of $K$ and $\d\theta_K$ is the {\em Cartan two-form}, non-degenerate (that is, $\wedge^n(\d\theta_K)\neq 0$) by the
regularity assumption. 
\end{defn}
This regularity assumption is equivalent to saying that the pair $(\d \theta_K, dt)$ defines a cosympletic structure \cite{CaNiYu2013}.  Then the map
$\flat_K:TE\rightarrow T^*E$ defined by $\flat_K(X)=X\hook\d\theta_K+\d t(X)\d t$ is a vector bundle isomorphism.

Here are some important aspects of the relationship of Cartan two-forms with \SODEs\:

\begin{propn}\label{NP2.5}\
\begin{itemize}

\item[Part 1] Let $\d\theta_K$ be a non-degenerate Cartan 2-form on $E$. Then there exists a unique \SODE\  $\Gamma_K$ with $\Gamma_K\hook\d\theta_K=0.$
    $\Gamma_K$ is the Euler-Lagrange field of $K.$

\item[Part 2] Let $\d\theta_K$ be a non-degenerate Cartan 2-form on $E$. Then $\d\theta_K(V_A,\Gamma)=0$ for any \SODE\  $\Gamma$, not just $\Gamma_K$.

\item[Part 3] Let $\Omega$ be a closed maximal rank 2-form on $E$ with $\Omega(V_1,V_2)=0$ for any pair of vertical vector fields, and suppose that
    $\Omega(V_A,\Gamma)=0$ for any \SODE\  $\Gamma.$ Then locally $\Omega=\d\theta_K$ for some fixed $K$ independent of the choice of $\Gamma.$
\end{itemize}
\end{propn}

Note that in part 3 the condition ``$\Omega(V_A,\Gamma)=0$ for {\em any} $\Gamma$" is well-posed since any two \SODEs\ differ only in their vertical components and
$\Omega(V_1,V_2)=0.$

\begin{proof}\
\begin{itemize}

\item[{\em Part 1}] The non-degenerate $\d\theta_K$ has a one-dimensional null space since $E$ is odd-dimensional. 
We will show that there exists exactly one \SODE\ in this null space. Using the standard identity and 
$\theta_K=\d K\circ S+K\d t$, we have for all $X\in\mathfrak{X}(E)$ and any \SODE\ $\Gamma$,
    \begin{align*}
    \d\theta_K(\Gamma,X)&=\Gamma(\theta_K(X))-X(\theta_K(\Gamma))-\theta_K([\Gamma,X])\\
    &=\Gamma(\theta_K(X))-X(K)-\theta_K([\Gamma,X]).
    \end{align*}
It is enough to take $X=\vf{u^A}$ and $X=\vf{x^A}$: in the former case the right hand side of the above expression is zero; in the latter case the right
hand side is
\[
\Gamma\left(\pd{K}{u^A}\right)-\pd{K}{x^A}
\]
and we see from \eqref{EL id} that $\Gamma\hook\d\theta_K=0$ if and only if $\Gamma$ is the unique Euler-Lagrange field for regular `Lagrangian' $K$, that is $\Gamma=\Gamma_K$. Observe that $\Gamma=\flat_K^{-1}(\d t)$.

\item[{\em Part 2}] For any \SODE\ $\Gamma$ the standard formula gives
\[
    \d\theta_K(V_A,\Gamma)=V_A(\theta_K(\Gamma))-\Gamma(\theta_K(V_A))-\theta_K([V_A,\Gamma]),
\]
and since $\theta_K=\d K\circ S+K\d t,$ the last term cancels the first in view of the bracket relations of section \ref{uncontr sodes}, and the middle term, being zero, gives the result.\newline

\item[{\em Part 3}] Because $\Omega$ is closed with $\Omega(V_1,V_2)=0$  locally we have $\Omega=\d\omega$ and so
    $\d\omega(V_1,V_2)=0$ for any pair of vertical vectors. And because the vertical distribution is integrable there exists locally a function $f$ on
    $E$ with $\omega=\d f$ on $V(E)$, that is, $\omega(V)=V(f)$ for all vertical fields $V$. Set $\omega_f:=\omega -\d f$ so that $\Omega=\d\omega_f$ and
    $\omega_f(V)=0.$ Since the difference of any pair of \SODEs\ is vertical we can set $K:=\omega_f(\Gamma)$ for all $\Gamma$. We show that
    $\theta_K=\omega_f.$ 

 Now $\theta_K(\Gamma) =K = \omega_f(\Gamma)$ since $\theta_K=K\d t+\d K\circ S$.  For arbitrary $\Gamma$
\[
\d K=\d((\omega_f(\Gamma))=\lie{\Gamma}\omega_f-\Gamma\hook\d\omega_f=\lie{\Gamma}\omega_f - \Gamma\hook\Omega.
\]
Since $S$ is vertical-valued, the assumption $\Omega(\Gamma,V)=0$ gives for any $X$
\[
\d K\circ S(X)=\lie{\Gamma}\omega_f(S(X))=-\omega_f(\lie{\Gamma}(S(X)).
\]
Recalling that $\lie{\Gamma}S(H)=-H$ for any $\Gamma$-horizontal field $H$ gives $\d K\circ S=\omega_f$ on horizontal fields and zero otherwise. Hence
$\theta_K=\omega_f$ independent of the choice of $\Gamma.$

\end{itemize}
\end{proof}

\begin{comments*}
\item[1.] The statement of Part 1 of the this proposition opens the question ``For a given \SODE\ $\Gamma$ is there a regular $K$ with $\Gamma \hook
    \d\theta_K=0?$" And the proof of Part 1 indicates a connection between this question and the inverse problem of the calculus of variations which we will address
    shortly.
\item[2.] Note that while part 3 indicates that $K$ is independent of the choice of \SODE\ in this case, part 1 of the proposition shows that a
    specific $\Gamma_K$ is nonetheless distinguished.
\item[3.] We can rephrase the identity \eqref{EL id} for the Euler-Lagrange field, by stating
    that if $\Gamma \hook\d\theta_L=0$ then $\Gamma$ is the unique Euler-Lagrange field for a regular Lagrangian $L$ .
\end{comments*}

We now further elaborate the structure of Cartan 2-forms with and without constraints before moving to the Lagrange-d'Alembert equations and the constrained and unconstrained inverse problems. The next result is similar to lemma 4 in \cite{CPT84} for the Euler-Lagrange field.

\begin{lem}\label{lemma 1}
Let $\Gamma$ be any \SODE\  on $E$ with associated horizontal distribution $\Sp\{H_A\}$, and let $K$ be any twice differentiable function on $E$. Then 
\begin{equation}
V_A(H_B(K))-V_B(H_A(K))=\onehalf\left(V_B(E_A(K))-V_A(E_B(K))\right)
\end{equation}
where we introduce the notation
\[
E_A(K):=\Gamma\left(\pd{K}{u^A}\right)-\pd{K}{x^A}.
\]
\end{lem}

\begin{proof}
Firstly, using the bracket relations we can write 
\[
H_A=\onehalf\left([V_A,\Gamma]+\vf{x^A}\right)
\]
so that
\begin{eqnarray*}
V_A(\Gamma(K))&=&[V_A,\Gamma](K)+\Gamma(V_A(K))\\
&=&2H_A(K)+E_A(K)\\
\implies V_B(H_A(K))&=&\onehalf(V_B(V_A(\Gamma(K)))-\onehalf V_B(E_A(K))\\
\implies V_A(H_B(K))-V_B(H_A(K))&=&\onehalf(V_B(E_A(K))-V_A(E_B(K))
\end{eqnarray*}
\end{proof}

\begin{lem}\label{lemma 2}
Let $\theta_K$ be any (regular) Cartan form and let $\Gamma$ be any \SODE\  on $E$, not necessarily the Euler-Lagrange field of $K,$ with associated
distributions $\Sp\{H_A\}$, $\Sp\{\psi^A\}$.  Then
\begin{equation}\label{Cartan form shape}
\d\theta_K=\beta_{AB}\theta^A\wedge\theta^B +E_A(K)\d t\wedge\theta^A + K_{AB}\psi^A\wedge\theta^B
\end{equation}
where $\beta_{AB}:=\oneqtr\left(V_A(E_B(K))-V_B(E_A(K))\right)$ and $K_{AB}:=V_AV_B(K)$.
\end{lem}

\begin{proof}
Use the identity for $\d\omega(X,Y),$ the bracket relations and $\theta_K:=\d K\circ S+ K\d t$ to establish that
\[
\d\theta_K(V_A,V_B)=0,\ \d\theta_K(\Gamma, V_A)=0,\d\theta_K(V_A,H_B)=K_{AB}.
\]
And
\begin{align*}
\d\theta_K(H_A,H_B)&=H_A(\theta_K(H_B)-H_B(\theta_K(H_A))-\theta_K([H_A,H_B])\\
&=H_A(V_B(K))-H_B(V_A(K))-0,\ \text{since}\ [H_A,H_B]\ \text{is vertical}\\
&=[H_A,V_B](K)+V_B(H_A(K)) - [H_B,V_A](K)-V_A(H_B(K))\\
&=V_B(H_A(K))-V_A(H_B(K)),\ \text{since}\ [H_A,V_B]=[H_B,V_A]\\
&=\onehalf\left(V_A(E_B(K))-V_B(E_A(K))\right),\ \text{by the lemma above}.
\end{align*}
Lastly,
\begin{align*}
\d\theta_K(\Gamma,H_A)&=\Gamma(V_A(K))-H_A(\theta_K(\Gamma))-\theta_K([\Gamma,H_A])\\
&=\Gamma(V_A(K))-H_A(K)-\theta_K(\Gamma^B_AH_B)\\
&=\Gamma(V_A(K))-\pd{K}{x^A}=E_A(K).
\end{align*}
\end{proof}

As a result of lemma \ref{lemma 2} and proposition \ref{NP2.5} part 3, we have
\begin{propn}\label{shape of Omega}
Let $\Omega$ be a closed maximal rank 2-form on $E$ with $\Omega(V_1,V_2)=0$ for any pair of vertical vector fields, and
    $\Omega(V_A,\Gamma)=0$ for a given \SODE\ $\Gamma$ with distributions $\Sp\{H_A\}, \Sp\{\psi^A\}.$ Then $\Omega$ has the form
\begin{equation}\label{Cartan form shape 2}
\Omega=\lambda_{AB}\theta^A\wedge\theta^B +\lambda_A\d t\wedge\theta^A + g_{AB}\psi^A\wedge\theta^B
\end{equation}
where $\lambda_{AB}:=\oneqtr\left(V_A(\lambda_B)-V_B(\lambda_A)\right)$ and $g$ is symmetric and non-degenerate.
\end{propn}

\begin{proof}
The proof of proposition \ref{NP2.5} part 3 is easily modified for a particular choice of $\Gamma$, the result being that the given $\Omega$ is a Cartan two-form,
$\d\theta_K$, for $K=\omega_f(\Gamma).$ Then lemma \ref{lemma 2} establishes the stated structure of $\Omega.$
\end{proof}

The expressions \eqref{Cartan form shape}, \eqref{Cartan form shape 2} provide an important model for the two-forms required in our inverse problems. 
Specifically, in the presence of a \SODE\ $\Gamma$ on $E$ any Cartan two-form will have the structure \eqref{Cartan form shape}. Conversely, in looking for 
Cartan two-forms with a specific relationship with $\Gamma$  we must restrict ourselves to two-forms with the structure \eqref{Cartan form shape 2}. The 
closure conditions on our candidates will be the generalised Helmholtz conditions, for example, the usual ones if $\Gamma$ must generate the kernel of the 
Cartan two-form we seek. 

For completeness and later use we give the closure (Helmholtz) conditions on the forms \eqref{Cartan form shape 2}:  
\begin{propn}\label{Omega HH}
The closure conditions on the two-forms of proposition \ref{shape of Omega} are:
\begin{align*}
&0=\d\Omega (\Gamma, V_A, V_B) = g_{AB}-g_{BA},\\
&0=\d \Omega (\Gamma, V_A, H_B) = \Gamma(g_{AB})-\Gamma^C_A g_{CB}-\Gamma^C_B g_{CA}-\onehalf(V_A(\lambda_B)+V_B(\lambda_A)),\\
&0=\d\Omega (\Gamma, H_A, H_B) = \Phi^C_B g_{CA}-\Phi^C_Ag_{CB}+2\Gamma(\lambda_{AB})+2\Gamma^C_A\lambda_{BC}-2\Gamma^C_B\lambda_{AC}+H_B(\lambda_A)-H_A(\lambda_B),\\
&0=\d \Omega (H_A, V_B, V_C) = V_C(g_{BA})-V_B(g_{CA}),\\
&0=\d\Omega(V_A,H_B,H_C) = 2V_A(\lambda_{BC})+H_C(g_{AB})-H_B(g_{AC})+V_A(\Gamma^D_B)g_{DC}-V_A(\Gamma^D_C)g_{DB},\\
&0=\d\Omega(H_A,H_B,H_C) = 2H_A(\lambda_{BC})+2H_C(\lambda_{AB})+2H_B(\lambda_{CA})-R^D_{BC}g_{DA}-R^D_{AB}g_{DC}-R^D_{CA}g_{DB}.
\end{align*}
The last two closure conditions are consequences of the others (see the appendix).
\end{propn}

In the next section we examine Cartan forms in the presence of constraints.

\subsubsection{The constrained case}\label{sect3.1.2}

We now examine the structure of $\d\PB\theta_K$ for an 
arbitrary, regular Cartan 1-form $\theta_K$ of definition \ref{Cartan form} and arbitrary constrained dynamics, $(\tGamma,G^\alpha)$. We also use the abbreviations 
$\tth_K:=\PB\theta_K$ and $\tth_{\bar{K}}:=\PB\theta_{\bar{K}}.$ In coordinates, $\theta_K=K\d t+\pd{K}{u^A}\theta^A$ with pull-back 
\begin{equation}
\tilde{\theta}_K:=\PB\theta_K=\bar{K}\d t+\left(\overline{\pd{K}{u^a}}+\pd{\Psi^\beta}{u^a}\overline{\pd{K}{u^\beta}}\right)\tth^a
+\overline{\pd{K}{u^\beta}}\teta^\beta,
\end{equation}
with the overline indicating pullback by $\PB$. Use of the chain rule:

\[
\pd{\overline K}{w}=\overline{\pd{K}{w}}+\overline{\pd{K}{u^\alpha}}\pd{\Psi^\alpha}{w},
\]
($w$ stands for any of $t,x^A,u^a$) gives
\begin{equation}\label{PBTheta_K}
\tilde{\theta}_K=\bar{K}\d t+\pd{\bar K}{u^a}\tth^a+\overline{\pd{K}{u^\beta}}\teta^\beta
= \bar{K}\d t+\d\bar{K}\circ\tS+\overline{\pd{K}{u^\beta}}\eta^\beta
\end{equation}
where $\tS=\tth^a\otimes\tV_a$ is the vertical endomorphism on $\tE.$ Notice that $\PB\theta_{\hat K}=\PB\theta_K$ where $\hat K=K+W\circ G^\alpha$ 
with $W'(0)=0;$ this will be important in our later discussion of uniqueness of Lagrangians. Notice also that $\bar{K}$ can be regarded as a singular Lagrangian on $E$ with $\PB\theta_{\bar{K}}=\tth_{\bar{K}}$ which will become important below when, in addition, $\pd{\bar{K}}{x^\alpha}=0$.

Following \eqref{PBTheta_K} we have, $\tilde{\theta}_K=\tilde{\theta}_{\bar{K}} +\overline{\pd{K}{u^\beta}}\eta^\beta,$ and  after some calculation,
\begin{equation*}
\d\tth_K=\d\tth_{\bar{K}}+\d\left(\overline{\pd{K}{u^\alpha}}\right)\wedge\teta^\alpha+\left(\overline{\pd{K}{u^\alpha}}\right)\d\teta^\alpha
\end{equation*}
with non-zero components in the presence of an arbitrary constrained \SODE\ $\tGamma$
\begin{align}
&\d\tth_K(\tGamma, \tH_a)=\d\tth_{\bar{K}}(\tGamma, \tH_a)-\overline{\pd{K}{u^\alpha}}\tilde{K}^\alpha_a,\\ 
&\d\tth_K(\tGamma, \vf{x^\alpha})=\d\tth_{\bar{K}}(\tGamma, \vf{x^\alpha})+\tGamma\left(\overline{\pd{K}{u^\alpha}}\right)+\overline{\pd{K}{u^\beta}}\pd{\Psi^\beta}{x^\alpha}, \\
&\d\tth_K(\tH_a,\tH_b)=\d\tth_{\bar{K}}(\tH_a,\tH_b)-\overline{\pd{K}{u^\alpha}}{\check R}^\alpha_{ab},\\ 
&\d\tth_K(\tH_a,\tV_b)=\d\tth_{\bar{K}}(\tH_a,\tV_b)+\overline{\pd{K}{u^\alpha}}\frac{\partial^2\Psi^\alpha}{\partial u^a\partial u^b},\label{kAB}\\
&\d\tth_K(\tH_a,\vf{x^\alpha})=\d\tth_{\bar{K}}(\tH_a,\vf{x^\alpha})+\tH_a\left(\overline{\pd{K}{u^\alpha}}\right)+\overline{\pd{K}{u^\beta}}\frac{\partial^2\Psi^\beta}{\partial u^a\partial x^\alpha}, \\
&\d\tth_K(\tV_a,\vf{x^\alpha})=\tV_a\left(\overline{\pd{K}{u^\alpha}}\right).
\end{align}

Furthermore,
\begin{align}
\d\tth_{\bar{K}}=&\left(\tGamma\left(\pd{\bar K}{u^a}\right)-\pd{\bar K}{x^a}-\pd{\Psi^\beta}{u^a}\pd{\bar K}{x^\beta}\right)\d t\wedge\tth^a
-\pd{\bar K}{x^\alpha}\d t\wedge\teta^\alpha\notag \\
&+\frac{\partial^2\bar{K}}{\partial u^a\partial u^b}\tpsi^a\wedge\tth^b -\frac{\partial^2 \bar{K}}{\partial x^\alpha\partial u^a}\tth^a\wedge\teta^\alpha
+\onehalf\left(\tH_a\left(\pd{\bar K}{u^b}\right)-\tH_b\left(\pd{\bar K}{u^a}\right) \right)\tth^a\wedge\tth^b.
\end{align}
These expressions indicate that for particular types of constraints there are specific Cartan 1-forms of interest.  For example, when the (affine) constraints are integrable, so that $\tilde{K}^\alpha_a=0$ and $\check R=0$ in the expressions above. More interestingly, in the non-integrable case with non-affine constraints and with $\tGamma, \Psi$, satisfying $\tilde{K}^\alpha_a=0$ and $\check R\neq 0$ (recall that affine constraints with $\tilde{K}^\alpha_a=0$ are integrable),
then all Cartan one-forms $\theta_K$ with $\pd{\bar K}{x^\alpha}=0$ satisfy
\begin{align}\label{bar K1}
&\d\tth_K(\tGamma, \tH_a)=\d\tth_{\bar{K}}(\tGamma, \tH_a),\\ 
&\d\tth_K(\tGamma, \vf{x^\alpha})=\tGamma\left(\overline{\pd{K}{u^\alpha}}\right)-\overline{\pd{K}{x^\alpha}}, \\
&\d\tth_K(\tH_a,\tH_b)=\d\tth_{\bar{K}}(\tH_a,\tH_b)-\overline{\pd{K}{u^\alpha}}{\check R}^\alpha_{ab},\\ 
&\d\tth_K(\tH_a,\tV_b)=\d\tth_{\bar{K}}(\tH_a,\tV_b)+\overline{\pd{K}{u^\alpha}}\frac{\partial^2\Psi^\alpha}{\partial u^a\partial u^b},\\
&\d\tth_K(\tH_a,\vf{x^\alpha})=\tH_a\left(\overline{\pd{K}{u^\alpha}}\right)+\overline{\pd{K}{u^\beta}}\frac{\partial^2\Psi^\beta}{\partial u^a\partial x^\alpha}, \\
&\d\tth_K(\tV_a,\vf{x^\alpha})=\tV_a\left(\overline{\pd{K}{u^\alpha}}\right),
\end{align}
with
\begin{equation}\label{bar K2}
\d\tth_{\bar{K}}=\tilde{E}_a(\bar{K})\d t\wedge\tth^a
+\frac{\partial^2\bar{K}}{\partial u^a\partial u^b}\tpsi^a\wedge\tth^b +\oneqtr\left(\tV_a\left(\tilde{E}_a({\bar K})\right)-\tV_b\left(\tilde{E}_b({\bar K})\right)\right)\tth^a\wedge\tth^b,
\end{equation}
where we have used the identity $\tH_a=\onehalf\left(\vf{x^a}+\pd{\Psi^\alpha}{u^a}\vf{x^\alpha}+[\tV_a,\tGamma]\right)$ and the obvious notation $\tilde{E}_a(\bar{K}):=\tGamma\left(\pd{\bar K}{u^a}\right)-\pd{\bar K}{x^a}.$ Expression \eqref{bar K2} bears close resemblance to \eqref{Cartan form shape} of lemma \ref{lemma 2}. It indicates, at least for this class of constraints and Cartan forms, that we can use the methods of the unconstrained inverse problem to find $\d\tth_{\bar{K}}$ and hence $\bar{K}$, see example \ref{xmpl4} in section \ref{sect5.2}.

There remains the nondegeneracy issue for $\d\tth_K$ and possibly $\d\tth_{\bar{K}}$. We will see in proposition \ref{Prop2.1} below that the uniqueness of the Lagrange-d'Alembert field in general requires only the nondegeneracy of the restriction of $\PB{\d\theta_K}$ to the annihilator of the constraint distribution $\Sp\{\teta^\alpha\}$. For this reason we will not give an extensive treatment of the full  nondegeneracy of $\PB{\d\theta_K},$ although it may be useful at a computational level in the constrained inverse problem.

 On the full space it is the maximal rank of an arbitrary Cartan form $\d\theta_K,$ namely $\wedge^n\d\theta_K\neq 0$, which expresses the assumed regularity of $K.$  We will briefly examine the implications of this condition on the various $p$-fold wedges of $\d\tth_K$ for $p<n.$ While these $2p$-forms will depend on the constraints, they are independent of any particular constrained \SODE\, so we will use $\hGamma_K$ of \eqref{hGamma}, so that 
\[
\PF\hGamma_K=\bar{\Gamma}_K-\overline{\Gamma_k(G^\alpha)}\tV_\alpha
\]
with its associated fields $\tH_a,\tV_b,\vf{x^\alpha}$ on $\tE.$ Recall that $\bar{\Gamma}_K$ coincides with the push-forward of a field on $\tE$ (namely $\hGamma_K$) if and only if $\Gamma_K$ is tangent to $\Img(\iE),$ that is, $\overline{\Gamma_k(G^\alpha)}=0.$

Notice first that, for $p$ with  $1+n+m<2p\leq 2n,$ every term in $\wedge^p\d\theta_K$ has $\d G^\alpha$ factors, so that all these exterior products pull-back to zero on $\tE$ (also by degree).  However, because lower order $p$-fold wedges of $\d\theta_K$ may have terms with no $\d G^\alpha$ factors, for the largest $p$ such that $2p\leq 1+n+m, $
$\wedge^p\PB\d\theta_K$ may be non-zero  since the $\text{dim}(\tE)=n+m+1.$ 

By contrast, the non-degeneracy condition on the restriction of $\PB\d\theta_K$ to the $2m+1$ dimensional distribution $\mathcal{D}:=\Sp\{\hGamma,\tH_a,\tV_a\}$ is  simply $\wedge^m(\PB\d\theta_K)|_\mathcal{D}\neq 0$ (see \cite{BatesSn} for a similar condition for the autonomous case).

\begin{defn}
We will say that  $\PB\d\theta_K$ is {\em non-degenerate or maximal rank} if $\wedge^p\PB\d\theta_K\neq 0$ for the largest $p$ such that $2p\leq 1+n+m$.

\end{defn} 
So there are two cases: the number of constraints, $n-m,$ is even (so $p=(n+m)/2$) or odd (so $p=(n+m+1)/2$), and we have (without proof)
\begin{propn}\label{PF regularity}
Given $n-m$ constraint functions,  $G^\alpha:=u^\alpha-\Psi^\alpha,$ a regular function $K$ and the existence of non-degenerate $\PB\d\theta_K$, then
\begin{itemize}
\item[1.] if $n-m$ is even, then $\wedge^{(n+m)/2}\PB{\d\theta_K}\neq 0$ and $\PB{\d\theta_K}$ has a one-dimensional null space
\item[2.] if $n-m$ is odd, then $\wedge^{(n+m+1)/2}\PB{\d\theta_K}\neq 0$ and $\PB{\d\theta_K}$ has trivial null space
\item[3.]  if $n-m$ is even, then the null space of $\PB{\d\theta_K}$ is spanned by a constrained \SODE\ $\tGamma$
if an only if $\Gamma_K$ is tangent to $\Img(\iE)$, in which case $\PF\tGamma=\bar{\Gamma}_K.$ 
\item[4.] If $n-m$ is odd there exists no non-trivial  $X$ on $\tE$ with $\PF{X}\in\Sp\{\bar{\Gamma}_K\}.$
\end{itemize}
\end{propn}

So this form of nondegeneracy puts very strong conditions on 
$\PB\d\theta_K$ and we revisit the issue after the next proposition with an example comparing the two degeneracy types.

\subsection{Geometric formulation of the Lagrange-d'Alembert equations}\label{sect3.2}

In order to treat the constrained inverse problem, we first give a geometric formulation of the Lagrange-d'Alembert equations akin to $\Gamma \hook
\d\theta_L=0$ for the Euler-Lagrange equations. In doing this we will have to account for the surprising fact that, in general, a given regular Lagrangian and specified constraints do not specify a unique Lagrange-d'Alembert field, that is, unique constrained dynamics, without additional regularity conditions. This obviously does not alter the inverse problem when the dynamics and constraints are given, although we might choose to remove some of the alternative Lagrangians because of this. Both the geometric formulation and degeneracy of the Lagrange-d'Alembert equations are discussed in \cite{SCS99} although our development is more direct and extensive.

We will need the array of values $\PB d\theta_L ({\tilde V}_a, {\tilde H}_b)$. Since $\tH_b=\vf{x^a}+\pd{\Psi^\alpha}{u^a}\vf{x^\alpha}-\tGamma^b_a\tV_a$, the expression $\PB d\theta_L ({\tilde V}_a, {\tilde H}_b)$ is independent of the vertical part of $\tH_b$ and hence of $\tGamma,$ and we will use it in the statement of the next proposition without ambiguity. It is also important to note that both $\eta^\alpha:=\theta^\alpha-\pd{\Psi^\alpha}{u^a}\theta^a$ and $\teta^\alpha:=\PB \eta^\alpha$ are contact forms, independent of any \SODEs.

\begin{propn}\label{Prop2.1}

Given specified constraints and a regular Lagrangian with the matrix $\PB d\theta_L ({\tilde V}_a, {\tilde H}_b)$ non-singular on $\Img(\iE)$, the Lagrange-d'Alembert field $\tGamma_L$ \eqref{tGamma LdA} can be uniquely determined from  the equation on $\tE$ 
\begin{equation}\label{LAeqn2}
\tGamma_L \hook \d(\PB\theta_L)=\lambda_\alpha\teta^\alpha.
  \end{equation}

\end{propn}

\begin{proof}
Suppose that $\tGamma_L$ is the Lagrange-d'Alembert field of $L$ and let $\Gamma$ be any \SODE\ on $E$ coinciding with $\PF\tGamma_L$ on the image of $\iE$ (in the notation of \eqref{hGamma} these are the \SODEs\  with $\hGamma=\tGamma_L.$) We will first show that $\tGamma_L \hook \d(\PB\theta_L)\in \Sp\{\teta^\alpha\}.$  As a result of \eqref{Cartan form shape} in Lemma \ref{lemma 2}, in
the bases provided by $\Gamma,$
\[
\d\theta_L=\beta_{AB}\theta^A\wedge\theta^B +E_A(L)\d t\wedge\theta^A + L_{AB}\psi^A\wedge\theta^B
\]
where $\beta_{AB}:=\oneqtr\left(V_A(E_B(L))-V_B(E_A(L))\right)$ and $L_{AB}:=V_AV_B(L)$. Hence
\[
\Gamma\hook\d\theta_L=E_A(L)\theta^A=(E_a(L)+E_\alpha(L)\pd{\Psi^\alpha}{u^a})\theta^a+E_\alpha(L)\eta^\alpha,
\]
so that  on the image of $\iE,$ using $\d(\PB\theta_L)=\PB(\d \theta_L)$ and $\PF \tGamma_L=\bGamma$,
\[
\tGamma_L \hook \d(\PB\theta_L)=\PF\tGamma_L\hook\d\theta_L=\lambda_\alpha\eta^\alpha \iff \lambda_\alpha=\overline{E_\alpha(L)} \ 
\text{and}\ \overline{E_a(L)}=-\lambda_\alpha\pd{\Psi^\alpha}{u^a}.
\]

We are left with verifying the statement about uniqueness. Assume that the equation (\ref{LAeqn2}) is satisfied by two 
distinct constrained \SODE s of the type \eqref{tGamma} (we allow distinct Lagrange multipliers for each). Then they differ by a vertical vector field 
$\tilde V = {\tilde V}^a {\tilde V}_a$, with 
\[
\PF\tilde V = {\tilde V}^a \left( V_a + \pd{\Psi^\alpha}{u^a} V_\alpha \right)\ \text{on}\ \Img(\iE),
\]
satisfying $\PF\tilde V \hook d\theta_L \in \Sp\{\eta^\alpha\}$. We show that under the assumed regularity condition, this requirement admits only $\tilde V=0$.

In the expression for $d\theta_L$ we may now use the bases provided by the Euler-Lagrange field of $L$, i.e. $\d\theta_L= g_{AB}\psi^A\wedge\theta^B$. In that case,  contracting with $\PF\tV$ and using $\{\theta^a,\eta^\alpha\}$ as a basis for the contact forms on $E$ gives,  

\begin{align}\label{V hook dtheta}
\PF\tilde V \hook d\theta_L &= {\tilde V}^a \left(d\theta_L(\PF\tV_a,\PF\tH_b)\theta^b+d\theta_L(\PF\tV_a,\PF\vf{x^\beta})\eta^\beta\right)\\ \notag
&= {\tilde V}^a \left(\left(\overline{g_{ab}} + \overline{g_{a\beta}} \pd{\Psi^\beta}{u^b} + \pd{\Psi^\alpha}{u^a}\left(\overline{g_{\alpha b}} + \overline{g_{\alpha \beta}} \pd{\Psi^\beta}{u^b}\right)\right) \theta^b + \left(\overline{g_{a\beta}}+\pd{\Psi^\alpha}{u^a}\overline{g_{\alpha \beta}}\right)\eta^\beta\right),
\end{align}
from which the statement follows (remembering that $\tH_b$ can belong to either of the \SODEs\ as noted in the preamble). However, notice that the non-singularity of $\PB d\theta_L ({\tilde V}_a, {\tilde H}_b)$  on $\Img(\iE)$ is generally only a sufficient condition for the $\tth^b$ components above to be zero, but it is a necessary condition for both the $\tth^b$ components and $\tV$ to be zero. 
\end{proof}

In \cite{SCS99} the  expression for $\d\tth_L(\tV_a,\tH_b):=\PB d\theta_L (\tV_a, \tH_b)$ in \eqref{V hook dtheta} is called the {\em $k$-matrix} and a regular $L$ will be termed {\em $k$-regular} at $p\in \tE$ if this array is non-singular at $p.$ Note from \eqref{kAB} that, for affine constraints $\d\tth_L(\tV_a,\tH_b)=\d\tth_{\bar{L}}(\tV_a,\tH_b),$ which is computationally useful.

A geometrical interpretation of the the $k$-matrix $\PB d\theta_L ({\tilde V}_a, {\tilde H}_b)$ is related to the restriction of $\PB d\theta_L$ to the $(2m+1)$--distribution on $\tilde{E}$ generated by
\[
{\mathcal D}:= {\rm Sp}\{\hat{\Gamma}, \tilde{H}_a, \tilde{V}_b\}=\{\tilde \eta^{\alpha}, m+1\leq \alpha\leq n\}^o
\]
where $\hat{\Gamma}$ is an arbitrary  constrained SODE as in (\ref{hGamma}). 
If we denote by $\tOmega_{\mathcal D}={\PB d\theta_L}\big|_{\mathcal D}$ then the non-degeneracy of $\tOmega_{\mathcal D}$, that is $\Lambda^m\tOmega_{\mathcal D}\not=0$, is equivalent to the non-degeneracy of the matrix with components $\PB d\theta_L ({\tilde V}_a, {\tilde H}_b)$. 
As a result the map
\[
\begin{array}{rrcl}
\flat_{\mathcal D}: &{\mathcal D}&\rightarrow& {\mathcal D}^*\\
& X&\rightarrow &X\hook\tOmega_{\mathcal D}+\d t(X)\d t
\end{array}
\]
is an isomorphism. And so the solution of the Lagrange-d'Alembert equations is unequivocally characterised  as the unique constrained \SODE\,    $\tilde{\Gamma}_L$ satisfying
\[
\tilde{\Gamma}_L\hook\tOmega_{\mathcal D}=0.
\]
Unfortunately $\tOmega_{\mathcal D}$ is not in general closed, however, it does provide a simple non-degeneracy check for the non-holonomic inverse problem similar to that in the unconstrained case.

We now provide the promised example contrasting the two forms of nondegeneracy.
\begin{xmpl}
Suppose that $L$ is a regular Lagrangian for an $n$-dimensional system with a single constraint ($m=n-1$) and Lagrange-d'Alembert field $\tGamma_L.$ Suppose that $\wedge^p\PB\d\theta_L\neq 0$ for $2p=n+m+1=2n$, so that $\d\tth_L:=\PB\d\theta_L$ is non-degenerate and by proposition \ref{PF regularity} has a trivial null space. Now $\tGamma_L\hook\d\tth_L=\lambda_{m+1}\teta^{m+1}$ means that 
\begin{align*}
&\wedge^n\d\tth_L(\tGamma_L,\tV_1,\dots,\tV_m,\dots,\tH_1,\dots,\tH_m,\vf{x^{m+1}})\\
&\ =\d\tth_L(\tGamma_L,\vf{x^{m+1}})|\d\tth_L(\tV_a,\tH_b)|\\
&\ =\lambda_{m+1}|k_{ab}|,
\end{align*}
meaning that this strong nondegeneracy requires {\em both} $\lambda_{m+1}\neq 0$ and $\wedge^m\tOmega_\mathcal{D}\neq 0.$
\end{xmpl}

However, in many cases the nondegeneracy on $E$ of the Lagrangian itself is not necessary to obtain a constrained regular problem.
\begin{xmpl}\label{sing_L xmpl}
 Consider the singular Lagrangian $L:  {\mathbb R}\times T{\mathbb R}^2\rightarrow {\mathbb R}$ given by 
$L=\frac{1}{2}\dot{x}^2+y$.
 Now introduce the  (integrable) constraint $\dot{y}=0$. Since
$$\theta_L=\dot{x}dx-\frac{1}{2}\dot{x}^2dt+ydt,$$
we get, in the
 coordinates $(t, x,y,\dot{x})$ on $\tilde{E},$
\[
\d\tth_L=d\dot{x}\wedge dx-\dot{x}d\dot{x}\wedge dt+dy\wedge dt.
\]
which is nondegenerate on $\tilde{E}$ (symplectic 2-form).
Observe that $\tilde{\eta}=dy$
and therefore
$i_X\d\tth_L=-dy$
gives us the unique constrained \SODE: 
\[
\tilde{\Gamma}_L=\frac{\partial}{\partial t}+\dot{x}\frac{\partial}{\partial x}
\]
Therefore, the system 
$\tilde{\Gamma}_L\hook\tOmega_{\mathcal D}=0$ has a unique \SODE\ solution.
\end{xmpl}
Nonetheless, we will pose our constrained inverse problem in section \ref{sect CIP} in terms of regular Lagrangians.

The following corollary will be useful later in looking for Lagrangians for constrained systems whose 
Lagrange-d'Alembert equations are Euler-Lagrange (that is, with zero Lagrange multipliers). 

\begin{cor}\label{zero multipliers}
Let $\Gamma_L$ be the Euler-Lagrange field of a regular Lagrangian $L$ and $\tGamma_L$ be the Lagrange-d'Alembert field of $L$ for some fixed constraints 
and Lagrange multipliers $\lambda_\alpha$.  Then, for all $\alpha,$
\[
\lambda_\alpha=0 \iff \PF\tGamma_L=\Gamma_L\ \text{on the image of}\ i_{\tE}. 
\]
\end{cor}

\begin{proof}

Suppose that $\PF\tGamma_L=\Gamma_L$ on the image of $\ i_{\tE},$ then 
\[
0=\PB(\Gamma_L\hook\d\theta_L)=\tGamma_L\hook\PB\d\theta_L=\lambda_\alpha\teta^\alpha,
\] 
so that all $\lambda_\alpha=0.$ Conversely, if all $\lambda_\alpha=0$ then 
\[
0=\tGamma_L\hook\PB\d\theta_L=\PB(\PF\tGamma_L\hook\d\theta_L).
\] 
Now $\PF\tGamma_L = \bGamma_L +V$ for some vertical $V,$ so that 
\[
0=\tGamma_L\hook\PB\d\theta_L=\PB (V\hook \d\theta_L) 
= \PB(g_{AB}\psi^A(V)\theta^B)
\]
in the 1-form basis for $\Gamma_L.$ Note that $g_{AB}\psi^A(V)\theta^B \notin \Sp\{\d G^\alpha\}$ so that $0=\overline{V\hook\d\theta_L}=\PF\tGamma\hook\d\theta_L.$ But $\d\theta_L$ has maximal rank on its domain which includes the image of $i_{\tE}$ and 
so $\PF\tGamma_L$ is (the unit multiple of) $\Gamma_L$ on the image of $i_{\tE}$. 

\end{proof}

We can utilise the last part of the proof of proposition \ref{Prop2.1} to directly relate 
the Euler-Lagrange field, $\Gamma_L$, to the Lagrange-d'Alembert field, $\tGamma_L$ and also give an alternate proof of the corollary above.
Following \eqref{hGamma} we define $\hGamma_L$ on $\tE$ by 
\[
\PF\hGamma_L=\Gamma_L-\overline{\Gamma_L(G^\alpha)}V_\alpha
\]
so that
\[
\hGamma_L=\vf{t}+u^a\vf{x^a}+\Psi^\alpha\vf{x^\alpha}+\overline F^a\vf{u^a}
\]
and
\[
\tGamma_L-\hGamma_L=(\tilde F^a-\bar F^a)\tV_a.
\]\

\begin{propn}\label{reaction force}
Let $\Gamma_L$ and $\tGamma_L$ be the Euler-Lagrange and Lagrange-d'Alembert fields respectively of a regular Lagrangian $L$ with Hessian $g_{AB}$ and constraint functions $G^\alpha:=u^\alpha-\Psi^\alpha.$ 
If the matrix $k_{ab}:=\PB d\theta_L ({\tilde V}_a, {\tilde H}_b)$ is invertible then 
\begin{equation}\label{tF-bF}
\tilde F^a-\bar F^a=\overline{\Gamma_L(G^\beta)}\left(\bar g_{\beta b}+\bar g_{\beta\sigma}\pd{\Psi^\sigma}{u^b}\right)k^{ab}
\end{equation}
with Lagrange multipliers
\begin{equation}
\lambda_\alpha=(\tilde F^a-\bar F^a)\left(\bar g_{a\alpha}+\pd{\Psi^\beta}{u^a}\bar g_{\alpha\beta}\right)-\overline{\Gamma_L(G^\beta)}\bar g_{\alpha\beta}
\end{equation}

\end{propn}

\begin{proof}
Set $\hat{V}:=\tGamma_L-\hGamma_L=(\tilde F^a-\bar F^a)\tV_a,$ then, on $\Img(\iE),$
\begin{align*}
\PF\hat{V}\hook\d\theta_L &=(\PF\tGamma_L-\Gamma_L+\overline{\Gamma_L(G^\alpha)}V_\alpha)\hook\d\theta_L\\
&=\lambda_\alpha\eta^\alpha-0+\overline{\Gamma_L(G^\alpha)}V_\alpha\hook\d\theta_L.
\end{align*}
Using \eqref{V hook dtheta} on the left hand side gives
\begin{equation}
(\tilde F^b-\bar F^b)k_{ab}=\overline{\Gamma_L(G^\beta)}\left(\bar g_{\beta a}+\bar g_{\beta\sigma}\pd{\Psi^\sigma}{u^a}\right)
\end{equation}
and
\begin{equation}
(\tilde F^a-\bar F^a)\left(\bar g_{a\alpha}+\pd{\Psi^\beta}{u^a}\bar g_{\alpha\beta}\right)=\lambda_\alpha+\overline{\Gamma_L(G^\beta)}\bar g_{\alpha\beta}
\end{equation}
being the $\theta^a$ and $\eta^\alpha$ components respectively; the $k$-regularity of $L$ then gives the result.
\end{proof}

In mechanical terms $\tilde F^b-\bar F^b$ are the components of the reaction forces induced by constraints. 
The presence of terms $\overline{\Gamma_L(G^\beta)}$ identifies the role of the possible tangency of $\Gamma_L$ to $\Img(\iE)$ or of the $G^\alpha$ as possible first integrals of $\Gamma_L.$ 

\begin{comment*}\ 

Because
\[
\teta^\alpha=\PB(\theta^\alpha-\pd{\Psi^\alpha}{u^a}\theta^a)
\]
equation \eqref{LAeqn2} is equivalent to
\begin{equation}
(\PF\tGamma)\hook \d\theta_L-\lambda_\alpha(\theta^\alpha-\pd{\Psi^\alpha}{u^a}\theta^a) \in \text{ker}(\PB).
\end{equation}

The left hand side of this relation is in the contact distribution and $\d(u^\alpha-\Psi^\alpha)$ is not, so we have
\begin{equation}
(\PF\tGamma)\hook \d\theta_L-\lambda_\alpha(\theta^\alpha-\pd{\Psi^\alpha}{u^a}\theta^a)=\ell_A\theta^A
\end{equation}
where $\ell_A\circ i_{\tE}=0.$
\end{comment*}

\section{The unconstrained Inverse Problem} \label{sec4}

\subsection{Generalities}

We start with a new, instrinsic proof of the Crampin, Prince and Thompson theorem \cite{CPT84} for the unconstrained inverse problem which avoids the a
priori use of the Helmholtz conditions. The theorem itself can be seen as Part 4 of proposition \ref{NP2.5} and the proof could easily be modified from
that of part 3; however the importance of the result warrants a stand-alone proof.

\begin{thm}\label{CPT Thm(1)} For a given \SODE\ $\Gamma$ and a closed, maximal rank 2-form $\Omega$ on $E$ with $\Gamma\hook\Omega=0$ and $\Omega(V_1,V_2)=0$
for any
pair of vertical vector fields, there exists a regular Lagrangian $L$ for which $\Gamma$ is the Euler-Lagrange field.
\end{thm}

\begin{proof}
We show that $\Omega = \d\theta_L$. Assume that we have a \SODE\ $\Gamma$ and a closed, maximal rank 2-form $\Omega$ on $E$ with $\Gamma\hook\Omega=0$ and $\Omega(V_1,V_2)=0$ for any pair of vertical vector fields. Now, since $\d\Omega=0$, $\Omega=\d\omega$ at least locally and
so $\d\omega(V_1,V_2)=0$ for any pair of vertical vectors. And because the vertical distribution is integrable there exists locally a function $f$ on $E$
with $\omega=\d f$ on $V(E)$, that is, $\omega(V)=V(f)$ for all vertical fields $V$. Now we will show that $\omega_f:=\omega-\d f$ is a Cartan 1-form,
$\theta_L,$ for $L:=\omega_f(\Gamma)=\omega(\Gamma)-\Gamma(f).$ \newline

Setting $\theta_L:=L\d t+\d L\circ S$, using $\d L=\d(\omega_f(\Gamma))=\lie{\Gamma}\omega_f,$ and the properties of $S$ including $\lie{\Gamma}S(V)=V$ for
$V\in V(E)$, it is straightforward to show that
\[
\d L\circ S = \omega_f\circ P_H
\]
and so $\theta_L=\omega_f$ as required.

Moreover, $\d\theta_L=\d\omega=\Omega$ so that $\Gamma \hook \d\theta_L=0, hence $ $L$ is a regular (local) Lagrangian with Euler-Lagrange field $\Gamma$.
\end{proof}

Notice that, given $\omega$ and $\d f$ with $V(f)=\omega(V),$ the Lagrangian $L$ can be directly recovered as $L=\omega(\Gamma)-\Gamma(f).$ If $f$ happens
to be a first integral of $\Gamma$ then the Lagrangian is just $\omega(\Gamma)$.

{\bf The Helmholtz Conditions}\\ How does theorem \ref{CPT Thm(1)} help solve the inverse problem as stated at the beginning of this section? In general,
one has a fixed $\Gamma$ but no $\Omega$. The modern approach is to find all closed, maximal rank two-forms $\Omega$ satisfying the conditions of the
theorem. But notice that proposition \ref{NP2.5} guarantees that such an $\Omega$ is a Cartan two-form $\d\theta_K$, so surely we are back in the position
of having to find a $K=L$ satisfying \eqref{trad EL}? The issue is resolved by applying the closure and maximal rank conditions last. So if a Lagrangian,
and hence a Cartan two-form, exists for which $\Gamma\hook\d\theta_L=0$ then we appeal to \eqref{Cartan form shape 2} of proposition \ref{shape of Omega}
to assume that our putative forms are $\Omega=g_{AB}\psi^A\wedge\theta^B$ with $g_{AB}=g_{BA}$ in the basis generated by the $\psi^A$ of our given
$\Gamma$. Applying closure and maximal rank to  two-forms in the ideal generated by $\{\onehalf(\psi^A\wedge\theta^B+\psi^B\wedge\theta^A)\}$ is guaranteed
to produce {\em only} Cartan two-forms, and all of them, and hence all Lagrangians for our $\Gamma$ or none; that is, to elaborate the uniqueness and
existence of Lagrangians for our \SODE.

These closure conditions are called the {\em Helmholtz conditions}, developed by Douglas \cite{D41}, in more modern form by Sarlet \cite{S82}.  In the 
current context we follow \cite{KP08}; for example, $\d\Omega(\Gamma, V_A,V_B)=0$ gives the symmetry condition required for the Hessian of the supposed
Lagrangian. In full, the four Helmholtz conditions are 
\begin{alignat*}{2}
&\d\Omega (\Gamma, V_A, V_B) = 0, \qquad \qquad &&\d \Omega (\Gamma, V_A, H_B) = 0, \\
&\d\Omega (\Gamma, H_A, H_B) = 0, \qquad \qquad &&\d \Omega (H_A, V_B, V_C)  = 0.
\end{alignat*}

The remaining conditions arising from $\d\Omega = 0$, namely
\begin{displaymath}
\d \Omega (H_A, H_B, V_C) = 0 \qquad \mbox{and} \qquad \d \Omega(H_A, H_B, H_C) = 0,
\end{displaymath}
are known to be derivable from the first four, e.g., in \cite{A03}, and a new proof is given in the appendix.

These produce the Helmholtz conditions of Sarlet on the $g_{AB}$, known as the `multiplier' or `multiplier matrix':

\[
g_{AB} = g_{BA}, \quad \Gamma(g_{AB}) =
g_{AC}\Gamma_B^C+g_{BC}\Gamma_A^C, \quad
g_{ac}\Phi_B^C  = g_{BC}\Phi_A^C, \quad \frac{\partial g_{AB}}{\partial u^C}  =
\frac{\partial g_{AC}}{\partial u^B},
\]

For reference:
\begin{align*}
\d\Omega & = (\Gamma (g_{AB}) - g_{CB} \Gamma^C_A - g_{AC} \Gamma^C_B) \d t \wedge \psi^A \wedge \theta^B \\
&+ (H_D (g_{AB}) - g_{CB} V_A (\Gamma^C_D))\psi^A\wedge\theta^B \wedge\theta^D \\
&+ V_C (g_{AB}) \psi^C \wedge\psi^A\wedge\theta^B \\
&+ g_{AB} \psi^A\wedge\psi^B \wedge \d t \\
&+ g_{CA} \Phi^C_B\theta^A \wedge\theta^B \wedge \d t \\
&+ g_{CA} H_B (\Gamma^C_D)\theta^A \wedge\theta^B\wedge\theta^D.
\end{align*}
To re-iterate: applying the closure condition last is the key to solving this inverse problem. We will produce the Helmholtz conditions for the constrained
case in the next section.

The following result (compare with lemma \ref{lemma 2} and proposition \ref{shape of Omega}) will be useful later.

\begin{propn}\label{shape of Omega 2} Suppose that $\Gamma_L$ is the Euler-Lagrange field of a Cartan two-form $\d\theta_L$ and let
$\Gamma$ be a \SODE\ distinct from $\Gamma_L.$ Let the modified force forms of $\Gamma_L$ and $\Gamma$ be $\hat\psi^A$ and $\psi^B$ respectively and
suppose that
\[
\d\theta_L=g_{AB}\hat\psi^A\wedge\theta^B
\]
where $g_{AB}$ satisfies the Helmholtz conditions (relative to $\Gamma_L$). Then
\begin{equation}\label{shape Omega eqn}
\d\theta_L=g_{AB}\psi^A\wedge\theta^B + g_{AB}\check F^A\d t\wedge\theta^B+\lambda_{AB}\theta^C\wedge\theta^B
\end{equation}
where $\check F^A:=F^A-\hat{F}^A$ are the (vertical) components of $\Gamma-\Gamma_L,$
\[
\lambda_{AB}:=\oneqtr\left(\pd{}{u^C}(g_{AB}\check F^A)-\pd{}{u^B}(g_{AC}\check F^A)\right)
\]
and
\begin{equation}\label{Gamma hook dtheta_L}
 \Gamma\hook\d\theta_L=g_{AB}\check F^A\theta^B.
\end{equation}
\end{propn}

\begin{proof}
Now
\[
\hat\psi^A=\d u^A-\hat{F}^A\d t+\hat{\Gamma}^A_B\theta^B,\quad\ \psi^A=\d u^A-F^A\d t+\Gamma^A_B\theta^B
\]
so that
\[
\d\theta_L=g_{AB}\hat\psi^A\wedge\theta^B=g_{AB}\psi^A\wedge\theta^B+g_{AB}\check F^A\d t\wedge\theta^B +
g_{AB}(\hat{\Gamma}^A_C-\Gamma^A_C)\theta^C\wedge\theta^B.
\]
Now
\begin{align*}
&\hat{\Gamma}^A_C-\Gamma^A_C=-\onehalf\pd{}{u^C}(\hat F^A-F^A)=\onehalf\pd{\check F^A}{u^C}\\
\implies &g_{AB}(\hat{\Gamma}^A_C-\Gamma^A_C)\theta^C\wedge\theta^B=\onehalf\left(\pd{(g_{AB}\check{F}^A)}{u^C}
-\pd{g_{AB}}{u^C}\check{F}^A\right)\theta^C\wedge\theta^B.
\end{align*}
But $g_{AB}$ satisfies the Helmholtz conditions, in particular, $\pd{g_{AB}}{u^C}=\pd{g_{AC}}{u^B}$, so that the final expression is 
skew symmetric on $B,C$, hence the claimed result. 

\end{proof}
As a result of lemma \ref{lemma 2} we have
\begin{cor}\label{Phys EL}
The \SODE\ $\Gamma$ of proposition \ref{shape of Omega 2} satisfies 
\[
E_A(L):=\Gamma\left(\pd{L}{u^A}\right)-\pd{L}{x^A}=\bar{F}_A:=g_{AB}\bar{F}^B.
\]
\end{cor}
As an aside, this resembles an equation of motion under a generalised  force $\bar{F}_A,$ which we briefly discuss in the  final outlook section. 
\subsection{The solution - an outline} \label{sec42}

While we are mainly concerned with the formulation of the inverse problem in the two cases, we should not omit a brief summary of the structure of the solution of the unconstrained inverse problem in the sense of Douglas and as elaborated by Do and Prince. The reader is directed to a variety of references with different degrees of technical detail \cite{D41,APST06,KP08,Do16,DP16,DP21}. 

The 2016 Douglas-Do classification hinges on the spectral properties of the Jacobi endomorphism $\Phi.$ In many of the cases the fine structure of the solution depends on the Frobenius integrability of particular eigenspaces, and this is examined with the Cartan-K\"{a}hler theorem via EDS which we won't discuss.

As indicated in the previous sub-section we begin with a generic $\Gamma$ and the ideal generated by $\{\onehalf(\psi^A\wedge\theta^B+\psi^B\wedge\theta^A)\}$ in which possible Cartan two-forms will be found. Then we lay out all the possible spectral types (diagonalisable, distinct eigenvalues, etc) for the Jacobi endomorphism of $\Gamma.$ For each type we make vertical and horizontal copies of the eigenforms of $\Phi$ and reframe our ideal accordingly. We also lay out all the possible integrability types for each of these double eigenspaces. The spectral type and the corresponding integrability types will dictate the largest differential (closed) sub-ideal of the original ideal in which any closed, maximal rank two-forms will be found, by an algebraic process of reduction. Importantly, non-existence of a solution arising at this differential ideal step usually comes from applying the non-degeneracy requirement.
So our classification is based on the diagonalisability of $\Phi$ firstly, then the number of distinct eigenvalues and the integrability of eigenform distributions of $\Phi$ and lastly, the step at which a `maximal rank' differential ideal is obtained. The uniqueness and existence details of the closed forms in each class are then determined by application of the Cartan-K\"{a}hler theorem.

As an example, here is an outline of the second of the four cases (\cite{Do16,DP16,DP21}):\\
{\bf Case B. $\mathbf{\Phi}$ is diagonalisable with distinct real-valued eigenvalues}\\
Further subcases will be divided according to the integrability of the lifted two-dimensional  eigenform distributions of $\Phi$ i.e. $q$ distributions are non-integrable and $n-q$ are integrable for $q=1,\dots,n$. For each $q,$ if up to and including the $q^{th}$ step of the algebraic reduction process there is no differential ideal, then there is no non-degenerate multiplier $g_{AB}$. Hence, for each $q$, the subcases to be considered are that a differential ideal is generated at step $1$, step $2$,..., up to step $q$.\\

\section{The constrained Inverse Problem\label{sect CIP}}\label{sec5}

\subsection{Generalities}

Before directly attacking the constrained inverse problem, it is worth forming some expectations about the solution. Since the given data, that is, 
constraints and a constrained {\SODE}, is entirely on $\tE,$ can we hope to construct a Lagrangian on the full space $E$? And, if we can, what can be said 
about its uniqueness? For example, consider a nontrivial  function $K$ on $E$ of the constraint functions $G^\alpha:=u^\alpha-\Psi^\alpha$ alone, 
zero-valued on the zero set of the $G^\alpha$ and having zero-valued exterior derivative on the constraint submanifold of $E$. The addition of such a 
function to a full-space Lagrangian does not change the Lagrange-d'Alembert equations but in general produces a different dynamical system on the full 
space, even on its restriction to the constraint submanifold on $E$ (as distinct from $\tE$). The point being that even if we can construct Lagrangians on the full space they 
may not give any insight into the physical nature of the external forces producing the constrained dynamics. 

As we saw in section \ref{motivation}, even the a priori restriction of possible solutions to physical Lagrangians such as \eqref{xmpl 1.1} on the full space does not produce a unique Lagrangian, with different solutions producing distinct, `pre-constraint' dynamics.
However, this is consistent with the view that the 
Lagrange-d'Alembert approach produces the constraint reaction forces to the externally applied ones. We will explore this ansatz in the next subsections.

In this discussion $\Gamma$ is any \SODE\ that is tangent to the constraints and that restricts on the constraints to a given \SODE-type 
vector field $\tilde\Gamma$. It is useful to think of $\tilde\Gamma$ as the nonholonomic vector field of a given physical 
nonholonomic system (whose integral curves are the nonholonomic trajectories), and of our inverse problem as the search for {\em alternative} Lagrangians 
for that mechanical system. We provide concrete examples in section \ref{dissipative}.

In what follows $\ker(\left.\PB\right|_{\Omega^1(\R,M)}):=\{\ell_A\theta^A\in\Omega^1(\R,M):\ell_A\circ i_{\tE}=0\},$ that is, the subset of contact 
one-forms which are zero-valued on $\tilde E$. 
\begin{thm} \label{CIP}
Suppose that we have a constrained system represented by a vector field $\tGamma$ given by \eqref{tGamma} and a closed, maximal rank 2-form $\Omega$ on $E$
with non-trivial pullback to $\tE$ satisfying
\begin{itemize}
\item[1.] $\Omega(V_A,V_B)=0$ for any pair of vertical vector fields,
\item[2.] $\Gamma \hook \Omega\in\Sp\{\eta^\alpha\}\mod \ker(\left.\PB\right|_{\Omega^1(\R,M)})$ for any \SODE\ $\Gamma$ with $\Gamma=\PF\tGamma$ on the 
    image of $i_{\tE},$ 
\item[3.] $\PB\Omega ({\tilde V}_a, {\tilde H}_b)$ is non-singular.
\end{itemize}
Then there exists a regular Lagrangian $L$ for which $\PB\Omega=\d(\PB\theta_L)$ so that $\tGamma$ is the unique Lagrange-d'Alembert field for $L$.\\

\end{thm}

\begin{proof}
Firstly, notice that condition 2 is equivalent to \eqref{LAeqn2} in proposition \ref{Prop2.1}. Secondly, condition 2 implies that $\Gamma 
\hook\Omega\in\Omega^1(\R,M)$ on $E$ and that this implies that $\Omega(\Gamma,V_A)=0$ for all such $\Gamma.$ Moreover, since all \SODEs\ differ only 
vertically, condition 1 implies that this is true for every \SODE. As a result of part 3 of proposition \ref{NP2.5}, conditions 1 and 2 above guarantee the 
existence of a Cartan two-form $\Omega=\d\theta_L$ for a putative Lagrangian $L.$ Then, as remarked, condition 2 gives 
\[
\tGamma \hook \PB\d\theta_L\in\Sp\{\teta^\alpha\}.
\]
The third condition ensures that the matrix $\PB d\theta_L ({\tilde V}_a, {\tilde H}_b)$ of $L$ is non-singular. So, proposition \ref{Prop2.1} establishes that $L$ is a Lagrangian with unique Lagrange-d'Alembert field $\tGamma$.
\end{proof}

Notice that the conclusion of the theorem only requires $\omega-\theta_L \in \ker \PB$, whereas by the construction of the proof it is closed (it is $\d f$).

{\bf Helmholtz Conditions for the constrained problem}

Using the unconstrained case as a model, we have $\tGamma$ and constraints $u^\alpha=\Psi^\alpha$, and we are looking for closed, maximal rank two-forms
$\Omega$ on $E$ modelled on \eqref{shape Omega eqn} where the distributions $\Sp\{H_A\},\Sp\{\psi^A\}$ belong to a particular \SODE\ $\Gamma$ which
coincides with $\PF\tGamma$ and $\Gamma\hook\Omega\in\Sp\{\eta^\alpha\}$ on the image of $\iE,$ that is

\begin{equation}\label{constr HH1}
\Omega={g}_{AB}\psi^A\wedge\theta^B + \lambda_B\d t\wedge\theta^B+\lambda_{DB}\theta^D\wedge\theta^B,
\end{equation}
\begin{align}\label{constr HH2}
&\Gamma\hook \Omega-\lambda_\alpha(\theta^\alpha-\pd{\Psi^\alpha}{u^a}\theta^a) \in \text{ker}(\PB),\notag \\
\text{equivalently}& \notag \\
&\Gamma\hook \Omega-\lambda_\alpha(\theta^\alpha-\pd{\Psi^\alpha}{u^a}\theta^a)=\ell_A\theta^A,
\end{align}
where $\lambda_{DB}:=\oneqtr\left(\pd{\lambda_B}{u^D}-\pd{\lambda_D}{u^B}\right),$ the $\lambda_\alpha$ are to be determined and $\ell_A\theta^A\in
\ker(\left.\PB\right|_{\Omega^1(\R,M)}).$

In these expressions $g_{AB}$ is invertible satisfying $g_{AB} = g_{BA}$ and $\frac{\partial g_{AB}}{\partial u^D}  = \frac{\partial g_{AD}}{\partial 
u^B}$, and, if such a Cartan two-form exists, the $\lambda_A$ will be the Euler-Lagrange expressions as described in lemma \ref{lemma 2}. 

We know from part 3 of proposition \ref{NP2.5} that all Cartan two-forms are of the type \eqref{constr HH1}, so it is the second condition
\eqref{constr HH2} which matters.
Hence we are looking for closed, maximal rank two-forms on $E$ in the ideal generated by
\[
\{\onehalf(\psi^A\wedge\theta^B+\psi^B\wedge\theta^A),(\lambda_B\d t+\lambda_{DB}\theta^D)\wedge\theta^B\ \text{(no sum on $B$)}\},
\]
of which there are infinitely many, but with conditions on $\tilde{E}$ 
\[
\lambda_b=-\lambda_\beta\pd{\Psi^\beta}{u^b}
\]
and $\frac{\partial g_{AB}}{\partial u^D}  = \frac{\partial g_{AD}}{\partial u^B}$ on the coefficient of the $\psi^A\wedge\theta^B$ term in $\Omega.$

Clearly \eqref{constr HH2} is a much weaker condition than $\Gamma\hook\d\theta_L=0$ of the unconstrained case, and some further ansatz may be required; we 
will explore this shortly.

Here are the Helmholtz conditions, that is, the closure conditions on $\Omega$ from proposition \ref{Omega HH} plus the $\tE$ condition. 

\begin{propn}\label{nonholHHcond}

For a fixed \SODE\ $\Gamma$ tangent to $\tE$ with horizontal fields $H_A$ and modified force forms $\psi^A$ the conditions on the form $\Omega$ above are
\begin{align*}
&0=\d\Omega (\Gamma, V_A, V_B) =g_{AB}-g_{BA}\ \text{satisfied by assumption},\\
&0=\d \Omega (\Gamma, V_A, H_B) = \Gamma(g_{AB})-\Gamma^C_A g_{CB}-\Gamma^C_B g_{CA}-\onehalf(V_A(\lambda_B)+V_B(\lambda_A)), \\
&0=\d\Omega (\Gamma, H_A, H_B) = \Phi^C_B g_{CA}-\Phi^C_A g_{CB}+2\Gamma(\lambda_{AB})+2\Gamma^C_A\lambda_{BC}-2\Gamma^C_B\lambda_{AC}+H_B(\lambda_A)-H_A(\lambda_B),\\
&0=\d \Omega (H_A, V_B, V_C) = V_C(g_{BA})-V_B(g_{CA})\ \text{satisfied by assumption},\\
&0=\d \Omega(V_A,H_B,H_C) = 2V_A(\lambda_{BC})+H_C(g_{AB})-H_B(g_{AC})+V_A(\Gamma^D_B)g_{DC}-V_A(\Gamma^D_C)g_{DB},\\
&0=\d \Omega(H_A,H_B,H_C) = 2H_A(\lambda_{BC})+2H_C(\lambda_{AB})+2H_B(\lambda_{CA})-R^D_{BC}g_{DA}-R^D_{AB}g_{DC}-R^D_{CA}g_{DB},\\
& \Gamma\hook \Omega-\lambda_\alpha(\theta^\alpha-\pd{\Psi^\alpha}{u^a}\theta^a)=\ell_A\theta^A.
\end{align*}
As before the last two closure conditions are redundant (see the appendix).
\end{propn}

Finally, we identify some possible additional ansatz on the \SODE\ $\Gamma$ to strengthen theorem \ref{CIP} and proposition \ref{nonholHHcond}. 
\begin{itemize}

\item[1.]In condition 2 of theorem \ref{CIP} we could restrict to \SODEs\ with first integrals $G^\alpha:=\dot{x}^\alpha-\Psi^\alpha$ which is a stronger condition than just tangency to 
    $\tE;$ this should help in concrete examples. 
\item[2.] We could restrict to \SODEs\ and Lagrangians $L$ with $V_A(E_B(L))=V_B(E_A(L))$ so that $\lambda_{AB}=0$ in the expression \eqref{constr HH1} for $\Omega.$ See section \ref{dissipative}.

\item[3.] Where viable, we could choose a $\Gamma$ which is the Euler-Lagrange field for at least one Lagrangian $L$, then if $\Omega$ satisfies the conditions of theorem \ref{CIP} so will $\Omega+\d\theta_L,$ generating more Lagrange-d'Alembert Lagrangians.

\end{itemize}

\subsection{Non-existence and existence}\label{sect5.2}

The question of whether a given constrained system admits a Lagrange-d'Alembert 
description at all is both more computationally difficult than its unconstrained counterpart and more ambiguous. 
As an example of this ambiguity, and in contrast to the example in section \ref{motivation}, two distinct physical Lagrangians with the same 
Euler-Lagrange equations, that is, the same unconstrained forces, can have distinct Lagrange-d'Alembert 
fields arising from the same constraint, as we will show. This is highly undesirable, at least in cases where the 
constraint can be switched on or off or modulated. Now there has been considerable discussion of `equivalent Lagrangians' in the unconstrained case (see, for example, \cite{MFLMR90,HS81,HH81}), and of avoidance 
of other types of ambiguity by various Lagrangian selection rules, group actions, etcetera (\cite{CP88,SP10}). 
For this paper we will simply restrict our existence question to the class of physical Lagrangians 
whose kinetic energy metric is exactly the `natural' metric of the underlying configuration manifold, even though it includes the cases in section \ref{motivation}. 
This should cover the many physical examples of non-holonomic systems.

\begin{xmpl}
This is the simplest example of the ambiguity described above. Consider the free particle, 
$\ddot x^A=0$, on $\R^n$ with the standard Euclidean metric, Lagrangian 
$L:=\onehalf\delta_{AB}\dot x^A\dot x^B.$ Applying an arbitrary non-holonomic constraint 
$\dot x^\alpha=\Psi^\alpha$, 
the constrained dynamics are given by the corresponding Lagrange-d'Alembert equations: 
\[
\ddot{x}^a=-\lambda_\alpha\pd{\Psi^\alpha}{\dot x^a},\ \lambda_\alpha:=\frac{d\Psi^\alpha}{dt},\ 
\dot x^\alpha=\Psi^\alpha.
\]

Now use the weighted metric $g_{AB}:=\text{diag}(1,\dots,1,2,\dots,2)$ with $m$ 1's and $(n-m)$ 2's. 
The corresponding, regular, kinetic energy Lagrangian, $L_g,$ still produces the free particle, however 
the Lagrange-d'Alembert equations for the same constraint are now
\[
\ddot{x}^a=-\mu_\alpha\pd{\Psi^\alpha}{\dot x^a},\ \mu_\alpha:=2\frac{d\Psi^\alpha}{dt}=2\lambda_\alpha,\
 \dot x^\alpha=\Psi^\alpha.
\]
Both Lagrangians have non-singular $k$ matrices for generic constraints.
\end{xmpl}

\begin{xmpl}
Another innocuous degeneracy in the unconstrained inverse problem arises from the addition of a 
degenerate Lagrangian for 
a given system to a regular one where the sum is regular. This has no impact on the 
Euler-Lagrange equations 
but should be avoided in the constrained case. We invite the reader to consider the 
regular Lagrangian on Euclidean $\R^3$
\[
L_0:=\onehalf\delta_{AB}\dot{x}^A\dot{x}^B - V_1(x)-V_2(y)-V_3(z)
\]
and the degenerate Lagrangian 
\[
L_{xy}:=\onehalf(\dot{x}^2+\dot{y}^2) - V_1(x)-V_2(y).
\]
It is straightforward to show that $L_0+L_{xy}$ is a regular Lagrangian for the Euler-Lagrange 
system of $L_0,$ however, for the arbitrary constraint $\dot{z}=\Psi$,
the Lagrange-d'Alembert fields are distinct (and $L_0$ and $L_0+L_{xy}$ both have non-singular $k$ matrices).
\end{xmpl}
\begin{comment*}
Perhaps the most notorious example of this last type is that of Henneaux and Shepley \cite{HS81}
 who showed that the Lagrangian 
\[
L_\gamma:=L+\frac{\gamma J}{r^2},\ L:=T-\mu/r,\ J:=\|\vec{r}\times\vec{v}\|,\ \gamma>0
\]
for the Coulomb problem produced a quantum mechanical hydrogen atom without energy level degeneracy. 
It might be interesting to see what 
impact the addition of the degenerate term has in classical constrained cases.

\end{comment*}

For these reasons we will restrict our search for constrained systems {\em without} a 
Lagrange-d'Alembert description to our class of canonical physical Lagrangians.

Now we turn to the important class of non-integrable, affine constrained systems admitting no 
physical Lagrangians on $\R^3$ with the Euclidean metric.

Consider a physical Lagrangian  
\begin{equation}\label{phys_L}
L=\onehalf\delta_{AB}\dot{x}^A\dot{x}^B+M_A\dot{x}^B-V
\end{equation}
where $M_A, V$ are functions of $t,x^A$ on $\R\times\R^3.$ Suppose that we have a constraint 
$\Psi$ affine in the $\dot{x}^a$ with coefficients on $\R\times\R^3$ and non-integrable.
We will assume that the $k$-matrix for $(L,\Psi)$ is non-singular, the condition for which is
\[
|k_{ab}|=\left(1+\left(\pd{\Psi}{\dot{x}}\right)^2\right)\left(1+\left(\pd{\Psi}{\dot{y}}\right)^2\right)
-\pd{\Psi}{\dot{x}}\pd{\Psi}{\dot{y}}\neq 0.
\]
It is straightforward to show that the Lagrange multiplier and the 
Lagrange-d'Alembert equations for the given constraint are
\begin{align*}
\lambda &=\frac{d}{dt}(\Psi+M_z)-\left(\pd{M_a}{z}\dot{x}^a+\pd{M_z}{z}\Psi\right)-\pd{V}{z}\\
\lambda\pd{\Psi}{\dot{x}^a}&=\frac{d}{dt}(\dot{x}^a+M_a)-\left(\pd{M_b}{x^a}\dot{x}^b
+\pd{M_z}{x^a}\Psi\right) - \pd{V}{x^a}
\end{align*}
We see that, apart from the acceleration terms, $\ddot{x}^a$, the Lagrange multiplier is 
quadratic in the $\dot{x}^a$ as are the two Lagrange-d'Alembert equations.
So in this non-degenerate case the accelerations are polynomial in the $\dot{x}^a$ of at 
most degree two.
So we have
\begin{propn}
Suppose that $\tGamma$ is a fixed, constrained \SODE\ on Euclidean $\R^3,$ where the constraint 
$\Psi$ is affine in the $\dot{x}^a$ with coefficients on $\R\times\R^3,$ and non-integrable. 
Then a necessary condition for $\tGamma$ to be the 
Lagrange-d'Alembert field of the physical Lagrangian \eqref{phys_L}, assumed to 
have non-degenerate $k$-matrix, is that the $\tilde{F}^a$ are polynomial in the $\dot{x}^a,$ 
with coefficients on $\R\times\R^3,$ of at most degree two.
\end{propn}

We close this section with an interesting class of non-holonomic inverse problems arising from expressions \eqref{bar K1} to \eqref{bar K2} for which solutions exist.

Given a pair $(\tGamma,\Psi)$ with $\Psi$ not affine and $\tilde{K}^\alpha_a=0,\check R\neq 0$ we could restrict the constrained inverse problem to finding Lagrangians $L$ with $\pd{\bar L}{x^\alpha}=0$ using \eqref{bar K1} to \eqref{bar K2}. Assuming  $\tGamma$ is Lagrange-d'Alembert for $L$ these expressions give

\begin{align}\label{bar K3}
&\d\tth_L(\tGamma, \tH_a)=\d\tth_{\bar{L}}(\tGamma, \tH_a)=0,\\ 
&\d\tth_L(\tGamma, \vf{x^\alpha})=\tGamma\left(\overline{\pd{L }{u^\alpha}}\right)-\overline{\pd{L}{x^\alpha}}=\lambda_\alpha, \\
&\d\tth_L(\tH_a,\tH_b)=-\overline{\pd{L}{u^\alpha}}{\check R}^\alpha_{ab},\\ 
&\d\tth_L(\tH_a,\tV_b)=\d\tth_{\bar{L}}(\tH_a,\tV_b)+\overline{\pd{K}{u^\alpha}}\frac{\partial^2\Psi^\alpha}{\partial u^a\partial u^b},\\
&\d\tth_L(\tH_a,\vf{x^\alpha})=\tH_a\left(\overline{\pd{L}{u^\alpha}}\right)+\overline{\pd{L}{u^\beta}}\frac{\partial^2\Psi^\beta}{\partial u^a\partial x^\alpha}, \\
&\d\tth_L(\tV_a,\vf{x^\alpha})=\tV_a\left(\overline{\pd{L}{u^\alpha}}\right),
\end{align}
with
\begin{equation}\label{bar K4}
\d\tth_{\bar{L}}=\frac{\partial^2\bar{L}}{\partial u^a\partial u^b}\tpsi^a\wedge\tth^b\ \text{and}\ \tGamma\left(\pd{\bar L}{u^a}\right)-\pd{\bar L}{x^a}=0.
\end{equation}
We see that $\bar{L}$, assumed independent of $x^\alpha$, can be determined by solving the `unconstrained inverse problem' for $\tGamma$ with the resulting construction of all $L$ following from the imposition of regularity and the remaining equations.

\begin{xmpl}\label{xmpl4}
On $\R^3$ consider the non-holonomic system $(\tGamma,\Psi)$ with  
\[
\dot z=\Psi:=z\dot{x}\dot{y}\ \text{and}\ \tGamma:=\vf{t}+\dot{x}\vf{x}+\dot{y}\vf{y}+\Psi\vf{z}
\]
(so that $\ddot{x}=0=\ddot{y}$). It is straightforward to check that $\tilde{K}^z_a=0,$ also that $\tGamma$ and, by the Helmholtz conditions or design, $\bar{L}:=\onehalf(\dot{x}^2+\dot{y}^2) +\onehalf(\dot{x}\dot{y})^2$ satisfy
\[
\tGamma\left(\pd{\bar{L}}{\dot{x}^a}\right)-\pd{\bar{L}}{x^a}=0.
\]
A second simple check shows that $\tGamma$ is the Lagrange-d'Alembert field (with $\lambda=0$) for
\[
L:=\onehalf\left(\dot{x}^2+\dot{y}^2+\left(\frac{\dot{z}}{z}\right)^2\right).
\]
The $k$-regularity of $\bar{L}$ and the regularity of $L$ are easily checked.

Note: this example was constructed by first fixing $\Psi$ and then using $\tilde{K}^z_a=0$ to determine $\tGamma.$ Finding a suitable $\bar{L}$ did not require the Helmholtz conditions because of the simplicity of $\tGamma;$ choosing an appropriate $L$ was also easy.
\end{xmpl}

\subsection{Relation to dissipative Lagrangian systems}  \label{dissipative}

We return to the Helmholtz conditions, as they were formulated for the constrained inverse problem in proposition~\ref{nonholHHcond}. As we already know, there may be  distinct, equivalent Lagrangians for a constrained system.  Because of their relative computational simplicity we will look first for Lagrangians which produce $\lambda_{AB} =0$. In a second step, we will consider those which have zero multipliers, $\lambda_A=0$ (as in corollary \ref{zero multipliers}).

If we set $\lambda_{AB} =0$ in the  Helmholtz conditions of proposition~\ref{Omega HH} (or proposition~\ref{nonholHHcond}), the  {resulting} conditions have already been studied in the context of (unconstrained) dissipative systems. The restriction $\lambda_{AB} =0$ means that $V_D(\lambda_B)=V_B(\lambda_D)$, 
 and therefore that there must exist some function $D$ on $E$ such that  $\lambda_A=V_A(D)$. 
The four non-redundant Helmholtz conditions then become 
\begin{align}
&0=g_{AB}-g_{BA}, \label{set2a}\\
&0= \Gamma(g_{AB})-\Gamma^C_A g_{CB}-\Gamma^C_B g_{CA}-V_AV_B(D), \\
&0= \Phi^C_B g_{CA}-\Phi^C_A g_{CB}  -H_A(V_B(D))+H_B(V_A(D))
\\
&0 = V_C(g_{BA})-V_B(g_{CA}) \label{set2b}
\end{align}
These conditions are exactly those that {also} appear in \cite{TWM} (see either equations (30-32) therein, or its Theorem~1). It is shown there that, when $(g_{AB})$ is  non-singular, 
the (unconstrained) \SODE\ $\Gamma$ is dissipative, in the sense that the differential equations that define its base integral curves can be brought in the form
\begin{equation} \label{dissEL}
\frac{d}{dt}\left(\pd{L}{\dot x^A}\right)-\pd{L}{x^A} = \pd{D}{{\dot x}^A},
\end{equation}
for some regular Lagrangian $L$ and {\em dissipative function} $D$. 

In fact, the unknown function $D$ can be eliminated from these conditions altogether by using the redundant Helmholtz 
ones: theorem~3 of \cite{TWM} states that \SODE\ $\Gamma$ is  dissipative if, and only if, there exists $g_{AB}$ satisfying 
\begin{align}
&0=g_{AB}-g_{BA},\label{set1a}\\
&0= V_C(g_{BA})-V_B(g_{CA}) \\
& 0 = H_C(g_{AB}) -H_B(g_{AC}) - g_{DB}V_A(\Gamma^D_C) + g_{DC}V_A(\Gamma^D_B)
  \\
&0= R^D_{BC}g_{DA}+R^D_{AB}g_{DC}+R^D_{CA}g_{DB} \label{set1b},
\end{align}
The last two of these are indeed the redundant conditions of proposition~\ref{Omega HH} with $\lambda_{AB}=0$.

In the examples below we will use these ideas to {obtain} a solution of the constrained inverse problem (for a \SODE\ $\Gamma$ that is tangent to the constraints and that restricts to the given nonholonomic vector field $\tilde\Gamma$). {First, we bring} the \SODE\ $\Gamma$ into the form of a dissipative Lagrangian system. {Second}, we check if it restricts to a nonholonomic Lagrangian system, by imposing 
\begin{equation} \label{final}
\Gamma\hook \Omega-\lambda_\alpha(\theta^\alpha-\pd{\Psi^\alpha}{u^a}\theta^a)=\ell_A\theta^A
\end{equation}
{at the very end}.

\begin{xmpl} The nonholonomic particle on $\R^3$. This instructive example has been extensively used throughout the literature, but it can originally be found in \cite{Rosenberg}.  The Lagrangian is  $L=\onehalf({\dot x}^2+{\dot y}^2+{\dot z}^2)$ and the constraint is $\dot z = -x\dot y$.
The nonholonomic vector field is 
\begin{equation} \label{nonholpartvf}
\tilde\Gamma= \pd{}{t}+{\dot x}\pd{}{x}+{\dot y}\pd{}{y}-x{\dot y}\pd{}{z}+0\pd{}{\dot x} - \left(\frac{x\dot x\dot y}{1+x^2}\right)\pd{}{\dot y} 
\end{equation}
and the corresponding Lagrange multiplier is $\lambda=-\frac{1}{1+x^2}{\dot x}{\dot y}$. 

We will look for alternative Lagrangians for this system. For that purpose, we will consider a whole class of  \SODEs\ with the properties that they are tangent to the constraint,  that they restrict to the 
nonholonomic vector field and that they are dissipative. For the quadratic \SODEs\ below, the first two conditions are clearly satisfied: 
\begin{eqnarray*}
\Gamma^1&=&\pd{}{t}+{\dot x}\pd{}{x}+{\dot y}\pd{}{y}+{\dot z}\pd{}{z}+
   \frac{(a_1{\dot x}+b_1{\dot y}+c_1{\dot z})({\dot z}+x{\dot y}) }{1+x^2}
\pd{}{\dot x} \\&& \hspace*{-15mm} + \left(\frac{(a_2{\dot x}+b_2{\dot y}+c_2{\dot z})({\dot z}+x{\dot y})}{1+x^2}-\frac{x\dot x\dot y}{1+x^2}\right)\pd{}{\dot y}
+\left( \frac{(a_3{\dot x}+b_3{\dot y}+c_3{\dot z})({\dot z}+x{\dot y})}{1+x^2}- \frac{\dot x\dot y}{1+x^2}\right)\pd{}{\dot z},
\end{eqnarray*}
 where all $a_i,b_i,c_i$ are constants, with $b_1\neq 0$. 

We will use the following strategy to find an alternative Lagrangian with the same Lagrange-d'Alembert  equations. First, we will select from our 
class of \SODEs\ those that can be brought in a dissipative form (i.e.\ we only consider the conditions of proposition~\ref{nonholHHcond}, where we have already assumed that $\lambda_{AB}=0$). To do so, we  explore the 
conditions (\ref{set1a}--\ref{set1b}) to find a suitable class of multipliers  $(g_{AB})$ and then we use conditions (\ref{set2a}--\ref{set2b}) as a set 
of determining equations for the dissipative function $D$. Once we know  which of the \SODEs\ in our class are dissipative, we finally impose the 
tangency condition on the Lagrange multiplier. According to  proposition~\ref{nonholHHcond}, the resulting multiplier 
$(g_{AB})$ provides a Lagrangian that solves our inverse problem. 

In full generality, this is a very difficult task and we will use an extra ansatz: given the shape of the standard
Lagrangian, we restrict to Lagrangians whose Hessian is a pseudo-Riemannian metric that is diagonal and invariant under translations of the
coordinates $y$ and $z$:
\begin{equation*} \label{gansatz}
g^1=\begin{bmatrix} g^1_{11}(x) &0&0\\0&g^1_{22}(x)&0\\0&0&g^1_{33}(x) \end{bmatrix}.
\end{equation*}

With this assumption, an easy calculation (we used Maple 2025 \cite{Maple2025}) shows that the conditions (\ref{set1a}--\ref{set1b}) can only be satisfied for
\[
g^1_{11} = C,  g^1_{22}= g^1_{33}=-b_1 C,
\]
($C$ is a non-zero constant) and only in the case where all constants are zero, except for $a_2=-1$ and for $b_1$ (which is otherwise arbitrary), i.e. for \SODEs\ of the
type
\begin{eqnarray*}
\Gamma^1&=&\pd{}{t}+{\dot x}\pd{}{x}+{\dot y}\pd{}{y}+{\dot z}\pd{}{z}+
   \frac{b_1{\dot y}({\dot z}+x{\dot y}) }{1+x^2}
\pd{}{\dot x}    - \frac{( {\dot z}+2x{\dot y}){\dot x}}{1+x^2}\pd{}{\dot y} - \frac{\dot x\dot y}{1+x^2} \pd{}{\dot z}.
\end{eqnarray*}

When we also assume that the dissipative function $D^1$ is independent of $y$ and $z$, conditions (\ref{set2a}--\ref{set2b}) are satisfied only if $D^1$ is of the type
\[
D^1= b_1C( x {\dot y}  + {\dot z})\frac{{\dot x}{\dot y}}{1+ x^2} + f(x){\dot x}+ A{\dot y}+ B{\dot z}  + g(x).
\]

 {Next, it can be shown that the only functions $L^1$ (with the above Hessian $g^1$) and $D^1$ that match the \SODEs\ of our class are}
\[
L^1= \onehalf C({\dot x}^2-b_1{\dot y}^2-b_1{\dot z}^2)+ \kappa_1(x)\dot x+ K_2\dot y+K_3\dot z+\kappa_4(x)
\]
and
\[
D^1= b_1C( x {\dot y}  + {\dot z})\frac{{\dot x}{\dot y}}{1+ x^2} -\kappa_4' (x){\dot x}+g(x).
\]
{In that case, using $\lambda^1_A=E_A(L^1)=V_A(D^1)$,}
\[
\lambda^1_1 = b_1C( x {\dot y}  + {\dot z})\frac{{\dot y}}{1+ x^2} -\kappa_4' (x) ,\quad \lambda^1_2= \frac{b_1C {\dot x}}{1+x^2}  \left(  2x  {\dot y}   + {\dot z}  \right),\quad \lambda^1_3=
b_1C\frac{{\dot x}{\dot y}}{1+ x^2},
\]
and thus indeed $\lambda^1_{AB}=0$. 

So far, we have only established that the \SODE\ is dissipative (on the full space). We still need to see which of the candidates satisfies the
last condition of our Proposition \ref{nonholHHcond}, namely (\ref{final}). This requires $\kappa_4(x)=K_4$. In the end, our alternative Lagrangians for the nonholonomic particle are (up
to a constant):
\[
L^1= \onehalf C({\dot x}^2-b_1{\dot y}^2-b_1{\dot z}^2)+ \kappa_1(x)\dot x+ K_2\dot y+K_3\dot z.
\]
The corresponding Lagrange multiplier is
\[
\lambda^1_3 = \left.\left(\frac{d}{dt}\left(\pd{L^1}{\dot z}\right)-\pd{L^1}{z}\right)\right|_{{\dot z}=-x{\dot y}}
  = \frac{Cb_1}{1+x^2}{\dot x}{\dot y}.
\]
 The standard Lagrangian $L$ of the nonholonomic particle is clearly one of these Lagrangians $L^1$, after setting $C=1$ and $b_1=-1$.  All the other Lagrangians $L^1$ have distinct and unrelated (unconstrained) Euler-Lagrange fields. It is the combination of both the Lagrangian $L^1$ with the appropriate multiplier $\lambda^1_3$ that returns the same nonholonomic vector field $\tilde\Gamma$ in \eqref{nonholpartvf}.
\end{xmpl}

\begin{xmpl}The vertically rolling disk. This example has been discussed in many papers and books. Here, we follow the notations as in the figure of \cite{Bloch03}, section 1.4. The Lagrangian of the  {nonholonomic} problem is
\[
L=\onehalf m ({\dot x}^2 + {\dot y}^2) + \onehalf I {\dot\theta}^2 + \onehalf J {\dot \phi}^2
\]
and the constraints are ${\dot x} = R \cos(\phi){\dot\theta}$ and ${\dot y} = R \sin(\phi){\dot\theta}$. The nonholonomic vector field is  
\[
\tilde\Gamma =\vf{t} + R \cos(\phi){\dot\theta} \pd{}{x} + R \sin(\phi){\dot\theta} \pd{}{y} + {\dot\theta} \pd{}{\theta} + {\dot\phi} \pd{}{\phi}   + 0 \pd{}{\dot\theta} + 0 \pd{}{\dot\phi}.
\]
To find alternative Lagrangians, we will use the parametrised class
\begin{eqnarray*}
\Gamma^2 &=&\pd{}{t}+{\dot x} \pd{}{x} + {\dot y} \pd{}{y} + {\dot\theta} \pd{}{\theta} + {\dot\phi} \pd{}{\phi} -R\sin(\phi)\dot\phi\dot\theta \pd{}{\dot x} 
+ R\cos(\phi)\dot\phi\dot\theta \pd{}{\dot y} \\ && + (a_1{\dot x} + b_1 {\dot y}+ c_1 {\dot\theta} + d_1 {\dot\phi})(\cos(\phi){\dot y} 
- \sin(\phi){\dot x}) \pd{}{\dot\theta}\\&& + (a_2{\dot x} + b_2 {\dot y}+ c_2 {\dot\theta} + d_2 {\dot\phi})(\cos(\phi){\dot y} 
- \sin(\phi){\dot x}) \pd{}{\dot\phi},
\end{eqnarray*}
whose members all restrict to $\tilde\Gamma$ on the constraint manifold.

Again, looking only for diagonal multipliers
\[
g^2=\begin{bmatrix} g^2_{11}(\phi) &0&0&0\\0&g^2_{22}(\phi)&0&0\\0&0&g^2_{33}(\phi)&0 \\ 0&0&0&g^2_{44}(\phi)  \end{bmatrix},
\]
one can show that the conditions (\ref{set1a}--\ref{set1b}) require zero constants except $d_1$ and $c_2$. In that case, we find

\[
g^2_{11} = g^2_{22} = \frac{C}{R}, \quad g^2_{33}= \frac{C}{d_1}, \quad g^2_{44}=\frac{C}{c_2},
\]
where  $C$ is an arbitrary, non-zero constant. This means our \SODEs\ are
\begin{eqnarray*}
\Gamma^2 &=&\pd{}{t}+ {\dot x} \pd{}{x} + {\dot y} \pd{}{y} + {\dot\theta} \pd{}{\theta} + {\dot\phi} \pd{}{\phi} -R\sin(\phi)\dot\phi\dot\theta \pd{}{\dot x} + R\cos(\phi)\dot\phi\dot\theta \pd{}{\dot y} \\ && +  d_1 {\dot\phi}(\cos(\phi){\dot y} - \sin(\phi){\dot x}) \pd{}{\dot\theta} +   c_2 {\dot\theta}  (\cos(\phi){\dot y} - \sin(\phi){\dot x}) \pd{}{\dot\phi}.
\end{eqnarray*}

This requires that the dissipative functions $D^2$ satisfying conditions (\ref{set2a}--\ref{set2b}) are
\[
D^2= C(\cos(\phi){\dot y} - \sin(\phi){\dot x}){\dot\theta}{\dot\phi}.
\]

The Lagrangians whose Hessian is $g^2$ and whose  dissipative equations return $\Gamma^2$ are then
\[
L^2=\onehalf \frac{C}{R}({\dot x}^2 + {\dot y}^2) + \onehalf \frac{C}{d_1} {\dot\theta}^2 + \onehalf \frac{C}{c_2}{\dot\phi}^2 + A {\dot x} + B {\dot y} + E {\dot\theta} + f(\phi){\dot\phi}.
\]
Finally, the Lagrange multipliers for the resulting nonholonomic are
\begin{eqnarray*}
\lambda^2_{x} &=& \left(\frac{d}{dt}\left(\pd{L_1}{\dot x}\right)-\pd{L_1}{x}\right)_{\tilde E}
  = -C\sin(\phi){\dot\theta}{\dot\phi} ,\\
\lambda^2_{y} &=& \left(\frac{d}{dt}\left(\pd{L_1}{\dot x}\right)-\pd{L_1}{y}\right)_{\tilde E} =  C\cos(\phi){\dot\theta}{\dot\phi}
.
\end{eqnarray*}
The  Lagrangian $L$ of the physical nonholonomic problem and its multipliers correspond to the case where $C=mR$, $d_1=\frac{mR}{I}$ and $c_2 = \frac{mR}{J}$. 
\end{xmpl}

\begin{comment*} In some examples,  we may be able to realise the  dissipative 
Helmholtz conditions (\ref{set2a})--(\ref{set2b})   
with a multiplier { $g_{AB}(t,x,\dot x)$} and with a dissipative function $D$ that only depends on $t$ and the coordinates $x^A$, i.e.\ $D=f(t,x)$. If so, we know from (\ref{dissEL}) that the \SODE\ $\Gamma$ is an Euler-Lagrange \SODE\  and so the Lagrange-d'Alembert equations will have zero multipliers,  $\lambda_A=0$  (see corollary \ref{zero multipliers}). This means that those integral curves that satisfy the constraints (amongst many others), will also satisfy the equations for the integral curves of $\tilde\Gamma$ (i.e.\ the 
nonholonomic trajectories of the given mechanical system). This type of questions have already appeared in the literature under the name {\em 
Hamiltonization of nonholonomic systems}, see e.g. \cite{BlFeMe09,BFMM15,TomMal}. In particular in \cite{BlFeMe09}, the (unconstrained) Helmholtz conditions 
have been used to find alternative Lagrangians. {These alternative Lagrangians} are therefore all examples for our {current} constrained inverse problem, and, by construction, they will have the property 
that the  Lagrange multipliers $\lambda_A$ (corresponding to the alternative Lagrangians) for the nonholonomic trajectories are zero. 
 \end{comment*}

\begin{xmpl} The nonholonomic particle (again). It has been shown in  \cite{BlFeMe09}  that the \SODE
\[
\Gamma^3 = \pd{}{t}+{\dot x}\pd{}{x}+{\dot y}\pd{}{y}+  {\dot z}\pd{}{z}+0\pd{}{\dot x} - \left(\frac{x\dot x\dot y}{1+x^2}\right)\pd{}{\dot y}
- \frac{\dot x\dot z}{x(1+x^2)}\pd{}{\dot z}
\]
 has the property that it restricts on the constraint to the vector field $\tilde\Gamma$ (in expression \eqref{nonholpartvf}) and that it has  the additional property that it is of (unconstrained) Euler-Lagrange 
type for the Lagrangians 
\[
L^3 =  \rho({\dot x}) + \onehalf\sqrt{1+x^2}\left(C_1\frac{{\dot y}^2}{\dot x} +C_2\frac{{\dot z}^2}{x\dot x}\right)
\]
(or: with dissipative function $D=0$). Here $\rho$ is any function with $\frac{d^2\rho}{d{\dot x}^2}\neq 0$ and $C_1,C_2$ are constants.
Indeed, one can easily compute the Lagrange-d'Alembert equations of these Lagrangians, obtaining the nonholonomic
particle with Lagrange multiplier $\lambda^3_1 = 0.$
\end{xmpl}

\section{Outlook}

In this paper we have addressed several aspects of the inverse problem of the calculus of variations. We have first investigated the general relationship between SODEs and Cartan two-forms in section~\ref{sec3}, both in the absence and in the presence of nonholonomic constraints. This allowed us to re-examine the unconstrained inverse problem (in section~\ref{sec4}) and to treat its extension to constrained systems on the same footing (in Section~\ref{sec5}). For both inverse problems, we are able to recast the problem in terms of a closed two-forms and derive appropriate  Helmholtz conditions. 

What we did not attempt for the constrained inverse problem is a Douglas-style classification of constrained systems, according to the existence and uniqueness of Lagrangians. This looks to be a very challenging problem, for the following reasons. Recall what we have said about the solution of the problem in section~\ref{sec42}. In the constrained case, we would need to consider the normal forms of the non-holonomic curvatures and the Jacobi endomorphism, both of which are on $\tE$ not $E$. When related back to properties of two-forms, this would involve the two-form $\tOmega=\PF\d\theta_L$ rather than  $\Omega=\d\theta_L$, which is now central in our formulation of the Helmholtz conditions (proposition~\ref{nonholHHcond}). In addition, the simplest nondegeneracy condition, $\wedge^m\tOmega_\mathcal{D}\neq 0,$ (where  $\mathcal{D}:=\Sp\{\hGamma,\tH_a,\tV_a\}$) is on the constraint submanifold $\tE$. This is also much simpler from an EDS point of view.  Then, a natural inverse problem would be: given a \SODE\ $\tilde{\Gamma}$ on $\tilde{E},$ study the conditions of  existence and uniqueness of a nondegenerate restricted 2-form 
$\tOmega_\mathcal{D}$ such that 
$
\tilde{\Gamma}\hook\tOmega_{\mathcal D}=0,
$
where 
 $\tOmega_{\mathcal D}={\PB d\theta_L}\big|_{\mathcal D}$. All the Lagrangians $L$ satisfying this condition (regular or not) would be solutions of the constrained inverse problem we posed in section \ref{sec3}. 
 
Next, we note that the key expression in the constrained inverse problem, 
$\tGamma\hook\tOmega=\lambda_\alpha\teta^\alpha$ bears a close resemblance to the unconstrained inverse problem condition when known external forces are present, $\Gamma\hook\Omega={F_A}^{ext}\theta^A.$ See also corollary \ref{Phys EL} and section~\ref{dissipative} on dissipative systems. A
 closer examination of the relationship between the Lagrange-D'Alembert equations and physical systems with other types of external forces   seems warranted. 
 
 Finally, the considerations about the regularity of the Lagrangian in sections~\ref{sect3.1.2},~\ref{sect3.2} lead to some interesting questions.  The Lagrangian $\bar{L}$ can be regarded as singular on $E$ but when will it have a regular $k$ matrix on $\tE$ as in example ~\ref{xmpl4}? And following example \ref{sing_L xmpl}, when does a singular Lagrangian on $E$ lead to a regular nonholonomic system on $\tE?$ 
The issue arising from both examples is whether a regular Lagrangian can be recovered from $\bar{L}$ or singular $L$, and how the required conditions are related to the nonholonomic curvatures and the constrained \SODE\ curvatures. 

 We postpone all these matters to a future paper.

\subsection*{Appendix: Redundancy in the Helmholtz conditions} We give a simplified proof of redundancy for the Helmholtz conditions in both the unconstrained and constrained cases.
An explicit original proof can be found in \cite{A03} and the result is implicit in \cite{CPT84}. 
\begin{propn}
For a \SODE\ $\Gamma$ with associated fields $\psi^A$ and $H_B$, and a two-form $\Omega$ on $E$ with $\Omega(V_A,V_B)=0$, the conditions
\[
\Gamma \hook\d\Omega=0,\ \d\Omega(H_A,V_B,V_C)=0
\]
imply
\[
\d\Omega(H_A,H_B,V_C)=0,\ \d\Omega(H_A,H_B,H_C)=0.
\]
\end{propn}

\begin{proof}
$\lie{\Gamma}\d\Omega=\d(\Gamma\hook\d\Omega)=0$ so that
\begin{align*}
\lie{\Gamma}(\d\Omega(H_A,V_B,V_C))&=\d\Omega([\Gamma,H_A],V_B,V_C)+\d\Omega(H_A,[\Gamma,V_B],V_C)+\d\Omega(H_A,V_B,[\Gamma,V_C])\\
&=\d\Omega(H_A,-H_B,V_C)+\d\Omega(H_A,V_B,-H_C)\\
&=-\d\Omega(H_A,H_B,V_C)+\d\Omega(H_A,H_C,V_B)
\end{align*}
using the bracket relations and the assumptions (note that $\Omega(V_A,V_B)=0$\\ implies $\d\Omega(V_A,V_B,V_C)=0$). Using $\d\Omega(H_A,V_B,V_C)=0$ and
cycling through the indices gives
\begin{align*}
&0=-\d\Omega(H_A,H_B,V_C)+\d\Omega(H_A,H_C,V_B),\\
&0=-\d\Omega(H_C,H_A,V_B)+\d\Omega(H_C,H_B,V_A),\\
&0=-\d\Omega(H_B,H_C,V_A)+\d\Omega(H_B,H_A,V_C).
\end{align*}
Hence
\begin{align*}
&\d\Omega(H_A,H_B,V_C)=\d\Omega(H_A,H_C,V_B)=-\d\Omega(H_C,H_B,V_A)=\d\Omega(H_B,H_A,V_C)\\
\implies &\d\Omega(H_A,H_B,V_C)=0.
\end{align*}
Establishing that $\d\Omega(H_A,H_B,H_C)=0$ follows from the assumptions is now straightforward:
\begin{align*}
&\lie{\Gamma}(\d\Omega(V_A,H_B,H_C))\\
&=\d\Omega([\Gamma,V_A],H_B,H_C)+\d\Omega(V_A,[\Gamma,H_B],H_C)+\d\Omega(V_A,H_B,[\Gamma,H_C])\\
&=\d\Omega(-H_A+\Gamma^D_AV_D,H_B,H_C)+\d\Omega(V_A,\Gamma^D_BH_D+\Phi^D_BV_D,H_C)+\d\Omega(V_A,H_B,\Gamma^D_CH_D+\Phi^D_CV_D)\\
&=-\d\Omega(H_A,H_B,H_C)
\end{align*}
using $\d\Omega(V_A,V_B,H_C)=0=\d\Omega(V_A,H_B,H_C)$. Hence the result.
\end{proof}

\subsection*{Acknowledgments}  The authors thank Marta Farr\'e Puiggal\'i for many discussions of the constrained inverse problem and David Saunders for his input on the geometry of constrained \SODEs.  GP acknowledges the Department of Mathematics at the University of Antwerp and Instituto de Ciencias Matem\'aticas (ICMAT) Madrid, for their warm hospitality during the preparation of this paper. TM thanks the Research Fund of the University of Antwerp (BOF) for its support through the DOCPRO projects 46954 and 49747. DMdD acknowledges financial support from the Grants PID2022-137909-NB-C21 and  the “Severo Ochoa Program for Centers of Excellence"  R\&D CEX2023-001347-S funded by MCIN/AEI / 10.13039/501100011033.

\vspace{5cm}

G.E.\,Prince \\
Department of Mathematical and Physical Sciences, La Trobe University,\\
Victoria 3086, Australia \\
Email: \url{g.prince@latrobe.edu.au}\\

T.\, Mestdag\\
University of Antwerp, Department of Mathematics,\\
Middelheimlaan 1, 2020 Antwerpen, Belgium \\
Email: \url{tom.mestdag@uantwerpen.be}\\

D.\,Mart\'{i}n de Diego\\
Instituto de Ciencias Matem\'aticas (CSIC-UAM-UC3M-UCM),\\
C/Nicol\'as Cabrera 13-15, 28049 Madrid, Spain\\
Email: \url{david.martin@icmat.es}

\end{document}